\documentclass[a4paper]
{cedram-smai-jcm}

\usepackage{lipsum}
\usepackage{amsfonts}
\usepackage{graphicx}
\usepackage{epstopdf}
\usepackage{algorithmic}
\usepackage{tikz-cd}
\usepackage{mathtools}
\usepackage{amsmath}
\usepackage{amssymb}
\usepackage{dsfont}
\usepackage{subcaption}

\newcommand{\eps}{\epsilon}
\newcommand{\Psip}{\Psi_{\scriptscriptstyle{PCA}}}
\newcommand{\Psin}{\Psi_{\scriptscriptstyle{NN}}}
\newcommand{\Psinum}{\Psi_{\scriptscriptstyle{num}}}
\newcommand{\ep}{e_{\scriptscriptstyle{PCA}}}
\newcommand{\en}{e_{\scriptscriptstyle{NN}}}
\newcommand{\OOmega}{D}
\newcommand{\tr}{\text{tr }} 
\newcommand{\X}{\mathcal{X}}
\newcommand{\Y}{\mathcal{Y}}
\newcommand{\R}{\mathbb{R}}
\newcommand{\E}{\mathbb{E}}
\newcommand{\HS}{\text{HS}(\mathcal{H})}
\newcommand{\Prob}{\mathbb{P}}
\newcommand{\mug}{\mu_{\text{G}}}
\newcommand{\mul}{\mu_{\text{L}}}
\newcommand{\mup}{\mu_{\text{P}}}

\definecolor{darkred}{rgb}{.7,0,0}

\definecolor{darkgreen}{rgb}{0,0.7,0} 
    
\definecolor{darkblue}{rgb}{0,0,0.7}

\definecolor{darkyellow}{rgb}{0.0,0.7,0.7}

\newcommand{\FX}{F_\X}
\newcommand{\FY}{F_\Y}
\newcommand{\GX}{G_\X}
\newcommand{\GY}{G_\Y}

\newcommand{\dX}{d_\X}
\newcommand{\dY}{d_\Y}

\newcommand{\s}{s}
\newcommand{\Lt}{t}
\newcommand{\Jr}{r}
\newcommand{\tx}{\tilde{x}}
\newcommand{\tchi}{\chi}

\newcommand{\HH}{\mathcal{H}}
\newcommand{\FH}{F_\HH}
\newcommand{\GH}{G_\HH}

\title{Model Reduction And Neural Networks For Parametric PDEs}

\author[K. Bhattacharya]{\firstname{Kaushik} \lastname{Bhattacharya}}
\address{Mechanical and Civil Engineering, California Institute of Technology, Pasadena, CA, USA}
\email{bhatta@caltech.edu}

\author[B. Hosseini]{\firstname{Bamdad} \lastname{Hosseini}}
\address{Computing and Mathematical Sciences, California Institute of Technology, Pasadena, CA, USA}
\email{bamdadh@caltech.edu}

\author[N. Kovachki]{\firstname{Nikola} \middlename{B.} \lastname{Kovachki}}
\address{Computing and Mathematical Sciences, California Institute of Technology, Pasadena, CA, USA}
\email{nkovachki@caltech.edu}

\author[A. Stuart]{\firstname{Andrew} \middlename{M.} \lastname{Stuart}}
\address{Computing and Mathematical Sciences, California Institute of Technology, Pasadena, CA, USA}
\email{astuart@caltech.edu}

  

\begin{document}

\begin{abstract}
We develop a general framework for data-driven approximation of
input-output maps between infinite-dimensional spaces. The proposed
approach is motivated by the recent successes of neural networks and deep learning, in combination
with ideas from model reduction. This combination results in a neural
network approximation which, in principle, is defined on infinite-dimensional
spaces and, in practice, is robust to the dimension of finite-dimensional
approximations of these spaces required for computation. For a class of 
input-output maps, and suitably  chosen probability measures on the inputs, 
we prove convergence of the proposed approximation methodology. 
We also include numerical experiments which 
demonstrate the effectiveness of the method, showing convergence and
robustness of the approximation scheme with respect to
the size of the discretization, and compare it
with existing algorithms from the literature; our examples
include the mapping from coefficient to solution in a divergence form 
elliptic partial differential equation (PDE) problem, 
and the solution operator for viscous Burgers' equation.
\end{abstract}

\maketitle

\section{Introduction}
\label{sec:intro}

At the core of many computational tasks arising in science and engineering
is the problem of repeatedly evaluating the output of an expensive 
forward model for many statistically similar inputs. Such settings include 
the numerical solution of parametric partial differential equations (PDEs),
time-stepping for evolutionary PDEs and, more generally, the evaluation of input-output
maps defined by black-box computer models. The key idea in this paper is
the development of a new data-driven emulator which is
defined to act between the infinite-dimensional
input and output spaces of maps such as those
defined by PDEs. By
defining approximation architectures
on infinite-dimensional spaces, we provide the basis
for a methodology which is robust to the resolution of the
finite-dimensionalizations used to create implementable algorithms.

This work is motivated by the recent empirical success of neural networks 
in machine learning applications such as image classification, aiming to 
explore whether this success has any implications for algorithm development in 
different applications arising in science and engineering. We further wish to compare 
the resulting new methods with traditional algorithms from the field of
numerical analysis for the approximation of infinite-dimensional maps, such 
as the maps defined by parametric PDEs or the solution operator for
time-dependent PDEs. We propose a method for approximation of such solution maps 
purely in a data-driven fashion by lifting the concept of neural networks to produce 
maps acting between infinite-dimensional spaces. 
Our method exploits approximate 
finite-dimensional structure in maps between Banach spaces
of functions through three separate steps:
(i) reducing the dimension of the input; (ii) reducing the dimension 
of the output, and (iii) finding a map 
between the two resulting finite-dimensional latent spaces.
Our approach takes advantage 
of the approximation power of neural networks while allowing for the use of 
well-understood, classical dimension reduction (and reconstruction) techniques. 
Our goal is to reduce the complexity of the input-to-output map by 
replacing it with a data-driven emulator. In achieving this goal 
we design an emulator which enjoys mesh-independent 
approximation properties, a fact which we establish through a combination
of theory and numerical experiments; to the best of our knowledge, these
are the first such results in the area of neural networks for PDE problems.

To be concrete, and to guide the literature review which follows,
consider the following prototypical parametric PDE
\[(\mathcal{P}_x y)(\s) = 0, \qquad \forall \s \in \OOmega,\]
where \(\OOmega \subset \R^d\) is a bounded open set, \(\mathcal{P}_x\) is a
differential operator depending on a parameter \(x \in \X\) and \(y \in \Y\)
is the solution to the PDE (given appropriate boundary conditions). The
Banach spaces
\(\X\) and \(\Y\) are assumed to be spaces of real-valued
functions on $\OOmega.$
Here, and in the rest of this paper, we  consistently
use $\s$ to denote the independent variable in spatially dependent PDEs,
and  reserve $x$ and $y$ for the input and output of the PDE model of interest.
We adopt this idiosyncratic notation (from the PDE perspective) to keep
our exposition in line
with standard machine learning notation for
input and output variables.

\begin{example}
Consider second order elliptic PDEs of the form
\begin{align}
\begin{split}
  \label{eq:darcy}
- \nabla \cdot (a(\s) \nabla u(\s)) &= f(\s), \quad \s \in \OOmega \\
u(\s) &= 0, \qquad \:\: \s \in \partial \OOmega
\end{split}
\end{align}
which are prototypical of many scientific applications.
As a concrete example of a mapping defined by this equation, we 
restrict ourselves to the setting where 
the forcing term \(f\) is fixed, and 
consider the diffusion coefficient \(a\) as the
input parameter \(x\) and the PDE solution \(u\) as output \(y\). 
In this setting, we have \(\X = L^\infty(D;\R_+)\), \(\Y = H_0^1(D;\R)\), 
and \(\mathcal{P}_x = - \nabla_s\cdot  ( a \nabla_s \cdot ) - f\), 
equipped with homogeneous Dirichlet boundary
conditions. This is the Darcy flow
problem which we consider numerically in Section \ref{sec:numdarcy}.
\end{example}

\subsection{Literature Review}
\label{subsec:litreview}

The recent success of neural networks on a variety of high-dimensional 
machine learning problems \cite{lecunnature} has led to a rapidly growing body of  
research pertaining to applications in scientific problems \cite{Adler2017,bhatnagar2019prediction,
  Cheng2019,weinandeepritz,gilmer2017neural,holland2019field,raissi2019physics,surrogatemodeling,smith2020eikonet}.
In particular, there is a substantial number
of articles which investigate the use of neural networks 
as surrogate models, and more specifically for obtaining the solution of 
(possibly parametric) PDEs.

We summarize the two most prevalent existing neural network
based strategies in the approximation of PDEs in general, 
and parametric PDEs specifically. The first approach can be 
thought of as image-to-image regression. The goal is to approximate 
the parametric solution operator mapping elements of $\X$  to  $\Y$.
This is achieved by discretizing both spaces to obtain finite-dimensional input and output spaces
of dimension $K$. We assume to have
access to  data in the form of observations of input \(x \) and output \(y\) discretized
on  
 \(K\)-points within the domain \(\OOmega\). The methodology then proceeds 
 by defining a neural network \(F: \R^K \to \R^K\) and regresses the input-to-output map
 by minimizing a misfit functional defined using the point values of $x$ and $y$ on
 the discretization grid.
The articles \cite{Adler2017,bhatnagar2019prediction,holland2019field,surrogatemodeling,geist2020numerical}
    apply this methodology 
for various   forward  and inverse problems in physics and engineering, 
utilizing a variety of  neural network architectures in the regression step;
the related paper \cite{khoo2017solving} applies a similar approach, but
the output space is $\R.$
This innovative set of papers demonstrate some success. However, from the
perspective of the goals of our work, their approaches are not robust to  
mesh-refinement: the neural network is defined as a mapping between 
two Euclidean spaces of values on mesh points.
The rates of approximation depend on the 
underlying discretization and an overhaul of the architecture would be
required to produce results consistent across different discretizations. 
The papers \cite{lu2019deeponet,lu2020deeponet} 
make a conceptual step in the direction
of interest to us in this paper, as they introduce an architecture based
on a neural network approximation theorem for operators 
from \cite{chen1995universal}; but as implemented
the method still results in parameters which depend on the mesh used.
Applications of this methodology may be found in
\cite{cai2020deepm,mao2020deepm,lin2020operator}.

The second approach does
not directly seek to find the parametric map from $\X$ to $\Y$
 but rather is  thought of, for fixed $x \in \X$, as being
a parametrization of the solution $y \in \Y$ by means of a
deep neural network  
\cite{Dockhorn19,weinandeepritz,hsieh2018learning,lagaris1998artificial,raissi2019physics,shin2020convergence}. 
This methodology parallels collocation methods for the numerical 
solution of PDEs
by searching over approximation spaces defined by neural networks.
The solution of the PDE is written as a neural network approximation
in which the spatial (or, in the time-dependent case, spatio-temporal) 
variables in $\OOmega$ are inputs and the solution is the output. 
This parametric function is then substituted 
into the PDE and the residual is made small
by optimization. The resulting neural network
may be thought of as a novel structure which composes the action of the
operator $\mathcal{P}_x$, for fixed $x$, with a neural network taking
inputs in $\OOmega$ \cite{raissi2019physics}.
While this method leads to an approximate solution map defined on the
input domain $\OOmega$ (and not on
a $K-$point discretization of the domain), 
the parametric dependence of the approximate solution map is fixed.
Indeed for a new input parameter \(x\), 
one needs to re-train the neural network by solving
the associated optimization 
problem in order to produce a new map \(y : \OOmega \to \R\); this may
be prohibitively expensive when parametric dependence of the solution is the target of analysis.
Furthermore 
the approach cannot be made fully data-driven as it needs knowledge of the underlying PDE,
and furthermore the operations required to apply the differential operator
may interact poorly with the neural network approximator during the back-propagation
(adjoint calculation) phase of the optimization.

The work \cite{ruthotto2019DeepNN} examines the
forward propagation of neural networks as the flow of a time-dependent PDE,
combining the continuous time formulation of ResNet 
\cite{haber2017stable,weinan2017proposal} with the idea
of neural networks acting on spaces of functions:
by considering the initial condition as a function, 
this flow map may be thought of as a neural network acting 
between infinite-dimensional spaces. The idea of learning
PDEs from data using neural networks, again generating
a flow map between infinite dimensional spaces,
was studied in the 1990s in the papers
\cite{krischer1993model,Kev98} with the former using
a PCA methodology, and the latter using the method of lines.
More recently the works \cite{HESTHAVEN201855,WANG2019289}
also employ a PCA methodology for the output space but only consider 
very low dimensional input spaces. Furthermore the works \cite{lee2020model,fresca2020comprehensive,gonzalez2018deep,fresca2020comprehensive} proposed a 
model reduction approach for dynamical systems by use of 
dimension reducing neural networks (autoencoders). However only a
fixed discretization of space is considered, yielding a method 
which does not produce a map between two infinite-dimensional spaces.

The development of numerical methods for parametric problems is not, of course,
restricted to the use of neural networks. Earlier works in the engineering literature
started in the 1970s focused on computational methods which represent PDE solutions
in terms of known basis functions that contain information
about the solution structure \cite{almroth1978automatic,nagy1979modal}.
This work led to the development of the  reduced basis method (RBM) which is widely adopted
in engineering; see 
\cite{barrault2004empirical,hesthaven2016certified,quarteroni2015reduced} 
and the references therein. 
The methodology was also used for stochastic
problems, in which the input space $\X$ is endowed with a probabilistic
structure, in \cite{boyaval2010reduced}. The study of RBMs led
 to broader interest in the approximation theory community focusing on  rates
of convergence for the RBM approximation of maps between Banach spaces,
and in particular maps defined through parametric dependence of
PDEs; see \cite{DeVoreReducedBasis} for an overview of this work.

Ideas from model reduction have been combined with data-driven
learning in the sequence of papers \cite{PEHERSTORFER2016196,mcquarrie2020datadriven,Benner_2020,peherstorfer2019sampling,qian2020lift}.
The setting is the learning of data-driven approximations to
time-dependent PDEs. Model reduction is used to find a low-dimensional
approximation space and then a system of ordinary differential 
equations (ODEs) is learned in this low-dimensional latent space.  
These ODEs are assumed to have vector fields from a known class
with unknown linear coefficients; learning is thus reduced to
a least squares problem. The known vector fields mimic properties
of the original PDE (for example are restricted to linear and
quadratic terms for the equations of geophysical fluid dynamics);
additionally transformations may be used to render the original PDE 
in a desirable form form (such as having only quadratic nonlinearities.)

The development of theoretical analyses to understand the use
of neural networks to approximate PDEs is currently in its infancy, but
interesting results are starting to emerge
\cite{herrmann2020deep,kutyniok2019theoretical,schwab2019deep,laakmann2020efficient}. 
A recurrent theme in the analysis of neural networks, and
in these papers in particular, is that the work typically
asserts the \emph{existence} of a choice of neural network parameters
which achieve a certain approximation property; because of the
non-convex optimization techniques used to determine the
network parameters, the issue of \emph{finding} these parameters
in practice is rarely addressed. 
Recent works take a different perspective on data-driven
approximation of PDEs, motivated by small-data scenarios;
see the paper \cite{cohen2020state} which relates, in part, to
earlier work focused on the small-data setting 
\cite{binev2017data,maday2015parameterized}.
These approaches are more akin to data assimilation
\cite{reich2015probabilistic,law2015data} where the data 
is incorporated into a model.

\subsection{Our Contribution}
\label{subsec:contribution}

The primary contributions of this paper are as follows:

\begin{enumerate}

\item we propose a novel data-driven methodology capable of  
learning mappings between Hilbert spaces;
  
\item the proposed method combines model reduction with neural networks to
obtain algorithms with controllable approximation errors as maps
between Hilbert spaces;

\item as a result of this approximation property of 
maps between Hilbert spaces, the learned maps exhibit 
desirable mesh-independence properties;

\item we prove that our architecture is
sufficiently rich to contain approximations of arbitrary
accuracy, as a mapping between function spaces; 

\item we present numerical experiments that demonstrate the efficacy
of the proposed methodology, demonstrate desirable mesh-indepence
properties, elucidate its properties beyond the confines 
of the theory, and compare with other methods for parametric PDEs.

\end{enumerate}

Section \ref{sec:method} outlines the approximation methodology, which
is based on use of {\it principal component analysis (PCA)} in a Hilbert space to finite-dimensionalize the
input and output spaces, and a neural network between the
resulting finite-dimensional spaces. 
Section \ref{sec:analysis} contains statement and proof of
our main approximation result, which invokes a global Lipschitz
assumption on the map to be approximated.  
In Section \ref{sec:numerics} we present our numerical experiments, some
of which relax the global Lipschitz assumption, and others which involve
comparisons with other approaches from the literature.
Section \ref{sec:conclusion} contains concluding remarks,
including directions for further study.
We also include auxiliary results in the appendix that complement and 
extend the main theoretical developments of the article.
Appendix~\ref{app:approxanalysis-local-Lipschitz} extends the analysis of
Section~\ref{sec:analysis} from globally Lipschitz maps to locally 
Lipschitz maps with controlled growth rates. Appendix~\ref{app:supportlemmas}
contains supporting lemmas that are used throughout the paper while
Appendix~\ref{app:b} proves an analyticity result pertaining to the
 solution map of the Poisson equation that is used in one of the 
numerical experiments in Section~\ref{sec:numerics}.

\section{Proposed Method}
\label{sec:method}

Our method combines PCA-based dimension reduction on the
input and output spaces $\X, \Y$ with a neural network that maps the dimension-reduced
spaces.  After a pre-amble in
Subsection \ref{sec:framework}, giving an overview of our approach, 
we continue in Subsection 
\ref{sec:functionalpca} with a description of PCA in the Hilbert space 
setting, including intuition about its approximation quality. 
Subsection \ref{sec:neuralnets} gives the background on neural networks
needed for this paper, and Subsection \ref{sec:representation} 
compares our methodology to existing methods.

\subsection{Overview}
\label{sec:framework}

Let \(\X\), \(\Y\) be separable Hilbert spaces
and \(\Psi: \X \to \Y\) be some, possibly nonlinear,
map. Our goal is to approximate \(\Psi\) from a finite collection of evaluations \(\{x_j,y_j\}_{j=1}^N\)
where \(y_j = \Psi(x_j)\). We assume that the $x_j$ are i.i.d. with respect to (w.r.t.) a probability measure  \( \mu\) 
supported on \(\X\). Note that with this notation 
the output samples $y_j$ are i.i.d. w.r.t. the
push-forward measure \(\Psi_\sharp\mu\). The approximation
of $\Psi$ from the data $\{x_j, y_j\}_{j=1}^N$ that we now develop
should be understood as being designed to be accurate with
respect to norms defined by integration
with respect to the measures $\mu$ and \(\Psi_\sharp\mu\) 
on the spaces \(\X\) and \(\Y\) respectively. 

Instead of attempting to directly approximate \(\Psi\), we first try to exploit 
possible finite-dimensional structure within the measures $\mu$ and $\Psi_\sharp \mu$. We accomplish this 
by approximating the identity mappings \(I_\X: \X \to \X\) and \(I_\Y: \Y \to \Y\) by a composition of two maps, 
known as the \textit{encoder} and the \textit{decoder} in the machine learning literature \cite{dimreduction,deeplearningbook},
which have finite-dimensional range and domain, respectively. We will then interpolate between the 
finite-dimensional outputs of the encoders, usually referred to as the \textit{latent codes}. Our approach is summarized in Figure \ref{fig:approach}.

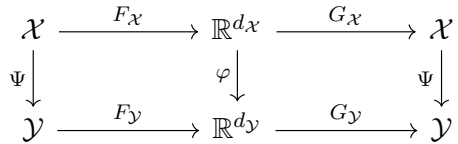
\begin{figure}
\centering
\begin{tikzpicture}
\begin{tikzcd}
\X \arrow[rr,"\FX"] \arrow[d,swap,"\Psi"] && \mathbb{R}^{\dX} \arrow[rr,"\GX"] \arrow[d,swap,"\varphi"] && \X \arrow[d,swap,"\Psi"] \\
\Y \arrow[rr,"\FY"] && \mathbb{R}^{\dY} \arrow[rr,"\GY"] && \Y
\end{tikzcd}
\end{tikzpicture}
\hspace*{6.2cm}
\caption{A diagram summarizing various maps of interest in our proposed approach for the approximation of input-output maps between infinite-dimensional spaces.} \label{fig:approach}
\end{figure}

Here, \(\FX\) and \(\FY\) are the encoders for the spaces $\X, \Y$ respectively, whilst \(\GX\) and \(\GY\) 
are the decoders, and \(\varphi\) is the map interpolating the latent codes. 
The intuition behind Figure \ref{fig:approach}, and, to some extent, the main focus of our analysis,
concerns the quality of the the approximations 
\begin{subequations}
\begin{align}
\GX \circ \FX  &\approx I_\X,   \label{eq:autoencodex} \\ 
\GY \circ \FY  &\approx I_\Y, \label{eq:autoencodey} \\  
\GY \circ \varphi \circ \FX   &\approx \Psi. \label{eq:approxpsi}
\end{align}
\end{subequations}
In order to achieve   \eqref{eq:approxpsi}
 it is  natural to choose $\varphi$ as  
\begin{equation}
\varphi:=\FY \circ \Psi \circ \GX; \label{eq:approxphi}
\end{equation}
then the approximation \eqref{eq:approxpsi} is limited only
by the approximations \eqref{eq:autoencodex}, \eqref{eq:autoencodey}
of the identity maps on $I_\X$ and $I_\Y$. We further label the approximation in \eqref{eq:approxpsi} by 
\begin{equation}
\Psip:=\GY \circ \varphi \circ \FX, \label{eq:apsipca}
\end{equation}
since we later choose PCA as our dimension reduction method. We note that \(\Psip\) is not used in practical computations  since $\varphi$ is generally 
unknown.  To make it practical 
we replace $\varphi$ with a data-driven 
approximation $\chi \approx \varphi$ obtaining,
\begin{equation}
\Psin:=\GY \circ \chi \circ \FX. \label{eq:apsipcn}
\end{equation}
Later we choose $\chi$ to be a neural network, hence the choice of the 
subscript $\scriptstyle{NN}$. 
The combination of PCA for the encoding/decoding 
along with the neural network approximation $\chi$ for $\varphi$, 
forms the basis of our computational methodology.

The compositions \(\GX \circ \FX\) and \(\GY \circ \FY\) are 
commonly referred to as \textit{autoencoders}.  There is a 
large literature on dimension-reduction methods 
\cite{originalpca,kernelpca,diffusionmaps,graphlaplaciandimreduce,dimreduction}
both classical and rooted in neural networks.
In this work, we will focus on PCA which is perhaps one of the simplest such methods known 
 \cite{originalpca}. We make this choice
due to its simplicity of implementation, excellent numerical performance on the
problems we  study in Section~\ref{sec:numerics}, and its amenability to analysis.
The dimension reduction in the input and output spaces is essential, as it allows
for function space algorithms that make use of powerful finite-dimensional
approximation methods, such as the neural networks we use here.

Many classical dimension reduction methods may be seen as encoders.
But not all are as easily inverted as PCA -- often there is 
no unambiguous, or no efficient, way to obtain the decoder.
Whilst neural network based methods such as deep 
autoencoders \cite{dimreduction} have shown empirical success
 in finite dimensional applications they currently lack 
theory and practical implementation in the setting of function spaces, and are 
therefore not currently suitable in the context of the goals of this paper.

Nonetheless methods other than PCA are likely to be useful within the
general goals of high or infinite-dimensional function approximation. Indeed, with PCA, we approximate the solution manifold (image space)
of the operator \(\Psi\) by the \textit{linear} space defined in equation \eqref{eq:pcasubspace}.
We emphasize however that, usually, \(\Psi\) is a nonlinear operator and our approximation 
succeeds by capturing the induced nonlinear input-output relationship within the latent codes by using a neural network.
We will show in Section \ref{sec:approxanalysis} that the approximation error of the linear space to the solution manifold goes to zero 
as the dimension increases, however, this decay may be very slow \cite{cohendevore2015,devore1998}. 
Therefore, it may be beneficial to construct nonlinear dimension reducing maps such as deep autoencoders on function spaces.
We leave this as an interesting direction for future work.

Regarding the approximation of $\varphi$ by neural
networks, we acknowledge  that there is considerable scope for the construction of the neural network, within different families
and types of networks, and potentially by using
other approximators. For our theory and numerics however we will focus on
 relatively constrained families of such networks, described in the following Subsection~\ref{sec:neuralnets}.

\subsection{PCA On Function Space}
\label{sec:functionalpca}

Since we will perform PCA on both $\X$ and $\Y$, and since PCA requires a Hilbert
space setting, the development here is in a generic real, separable
Hilbert space  \(\mathcal{H}\)  
with inner-product and norm denoted by \(\langle \cdot, \cdot \rangle\) 
and \(\|\cdot\|\) respectively.
We let \(\nu\) denote a probability measure
supported on \(\mathcal{H}\), and make the assumption of 
a finite fourth moment: \(\E_{u \sim \nu} \|u\|^4 < \infty\).  We denote by 
\(\{u_j\}_{j=1}^{N}\) a finite collection of $N$ i.i.d. draws from \(\nu\) that
will be used as the training data on which PCA is based. 
Later we apply the PCA methodology 
 in two distinct  settings where the space $\mathcal H$ is taken to be the input space $\X$ and the data 
 \(\{u_j\}\) are the input samples \(\{x_j\}\) drawn from the input measure $\mu$, or $\mathcal{H}$ is taken to be  the output space $\Y$ and the data $\{u_j\}$
 are the corresponding outputs
\(\{y_j = \Psi(x_j)\}\) drawn from the push-forward measure $\Psi_\sharp \mu$.
The following exposition, and the subsequent analysis 
in Section \ref{sec:pcaanalysis}, largely follows the works 
\cite{Blanchard2007,ShaweTaylor2,ShaweTaylor1}. We will consider the standard
version of non-centered PCA, although more sophisticated versions such as 
kernel PCA have been widely used and analyzed \cite{kernelpca} and could be
of potential interest within the overall goals of this work.  We choose to work 
in the non-kernelized setting as there is an unequivocal way of producing the decoder.

For any subspace \(V \subseteq \mathcal{H}\), denote by
\(\Pi_V : \mathcal{H} \to V\) the orthogonal projection operator and define the \textit{empirical projection error},
\begin{equation}
\label{eq:empiricalerror}
R_N(V) \coloneqq \frac{1}{N}\sum_{j=1}^N \|u_j - \Pi_V u_j\|^2.
\end{equation}
PCA consists of projecting the data onto a finite-dimensional subspace of \(\mathcal{H}\) for which this
error is minimal. To that end, consider the \textit{empirical, non-centered covariance} operator
\begin{equation}
\label{eq:empiricalcovariance}
C_N \coloneqq \frac{1}{N} \sum_{j=1}^N u_j \otimes u_j
\end{equation}
where \(\otimes\) denotes the outer product. It may be shown that 
\(C_N\) is a non-negative, self-adjoint, trace-class operator on 
$\mathcal{H}$, of rank at most \(N\) \cite{zeidler2012appliedfunc}. Let 
\(\phi_{1,N}, \dots \phi_{N,N}\) denote the eigenvectors of 
\(C_N\) and \(\lambda_{1,N} \geq \lambda_{2,N} \geq \dots \geq \lambda_{N,N} \geq 0\) 
its corresponding eigenvalues in decreasing order. Then for any $d \ge 1$ we define the \textit{PCA subspaces}
\begin{equation} 
\label{eq:pcasubspace}
V_{d,N} = \text{span} \{\phi_{1,N},\phi_{2,N},\dots,\phi_{d,N}\} \subset \mathcal{H}.
\end{equation}
It is well known \cite[Thm. 12.2.1]{murphybook} that \(V_{d,N}\) solves the minimization problem
\[\min_{V \in \mathcal{V}_d} R_N(V),\]
where \(\mathcal{V}_d\) denotes the set of all \(d\)-dimensional subspaces of \(\mathcal{H}\). Furthermore 
\begin{equation}
\label{eq:nt}
R_N(V_{d,N}) = \sum_{j=d+1}^N \lambda_{j,N},
\end{equation}
hence the approximation is controlled by the rate of decay of the spectrum of \(C_N\).

With this in mind, we define the \textit{PCA encoder}
\(\FH: \HH \to \R^d\) as the mapping from $\HH$ 
to the coefficients of the orthogonal projection onto \(V_{d,N}\) namely,
\begin{equation}
\label{eq:encoder}
\FH(u) = (\langle u, \phi_{1,N} \rangle, \dots, \langle u, \phi_{d,N} \rangle)^T \in \mathbb R^d.
\end{equation}
Correspondingly, the PCA decoder \(\GH: \R^d \to \HH\) constructs an 
element of \(\HH\) by taking as its input the coefficients constructed by \(\FH\) and forming an expansion
in the empirical basis by zero-padding the PCA basis coefficients, that is 
\begin{equation}
\label{eq:decoder}
\GH(\s) = \sum_{j=1}^d \s_j \phi_{j,N} \qquad \forall \s \in \R^d.
\end{equation}
In particular,
\begin{equation*}
(\GH \circ \FH)(u)  = \sum_{j=1}^d \langle u, \phi_{j,N} \rangle \phi_{j,N},\quad
\text{equivalently}\quad
\GH \circ \FH  = \sum_{j=1}^d \phi_{j,N} \otimes \phi_{j,N}.
\end{equation*}
Hence \(\GH \circ \FH = \Pi_{V_{d,N}}\), a \(d\)-dimensional approximation to the
identity \(I_{\HH}\).

We will now give a qualitative explanation of this 
approximation to be made quantitative in  Subsection \ref{sec:pcaanalysis}.
It is natural to consider minimizing the 
infinite data analog of \eqref{eq:empiricalerror}, namely the \textit{projection error}
\begin{equation}
\label{eq:projectionerror}
R(V) \coloneqq \E_{u \sim \nu}\|u - \Pi_V u\|^2,
\end{equation}
over $\mathcal{V}_d$ for $d \ge 1$. 
Assuming \(\nu\) has a finite second moment, there exists a unique, self-adjoint,
non-negative, trace-class operator \(C : \HH \to \HH\) termed the \textit{non-centered covariance} such that
$\langle v, Cz \rangle = \E_{u \sim \nu}[\langle v, u \rangle \langle z, u \rangle]$, $\forall v,z \in \HH$
(see \cite{Baxendale}).
From this, one readily finds the form of \(C\) by noting that
\begin{equation}
\label{eq:covaraince}
\langle v, \E_{u \sim \nu}[u \otimes u] z \rangle = \E_{u \sim \nu} [\langle v, (u \otimes u)z \rangle]
= \E_{u \sim \nu}[\langle v, u \rangle \langle z, u \rangle], 
\end{equation}
implying that
$C = \E_{u \sim \nu}[u \otimes u].$
Moreover, it follows that
\[\tr C = \E_{u \sim \nu} [\tr u \otimes u] = \E_{u \sim \nu} \|u\|^2 < \infty.\]

Let \(\phi_1,\phi_2,\dots\) denote the eigenvectors of \(C\) and \(\lambda_1 \geq \lambda_2 \geq \dots\) the
corresponding eigenvalues. In the infinite data setting $(N = \infty)$ it is natural to think of $C$
and its first $d$ eigenpairs as known.
We then define the \textit{optimal projection space}
\begin{equation}
\label{eq:optprojectionspace}
V_d = \text{span} \{\phi_1, \phi_2, \dots, \phi_d\}.
\end{equation} 
It may be verified that \(V_d\) solves the minimization problem
$\min_{V \in \mathcal{V}_d} R(V)$
and that 
$R(V_d) = \sum_{j=d+1}^\infty \lambda_j.$

With this infinite data perspective in mind observe that PCA makes the approximation 
\(V_{d,N} \approx V_d\) from a finite dataset. The approximation 
quality of \(V_{d,N}\) w.r.t. \(V_d\) is related to the approximation quality
of \(\phi_j\) by \(\phi_{j,N}\) for \(j=1,\dots,N\) and therefore to the approximation 
quality of \(C\) by \(C_N\). Another perspective is via the Karhunen-Loeve 
Theorem (KL) \cite{Lord}. For simplicity, assume that \(\nu\) is
mean zero, then \(u \sim \nu\) admits an  expansion of the form
$u = \sum_{j=1}^\infty \sqrt{\lambda_j} \xi_j \phi_j$
where \(\{\xi_j\}_{j=1}^\infty\) is a sequence of scalar-valued, mean zero, pairwise uncorrelated random variables.
We can then truncate this 
expansion and make the approximations
\[
u \approx \sum_{j=1}^d \sqrt{\lambda_j} \xi_j \phi_j \approx \sum_{j=1}^d \sqrt{\lambda_{j,N}} \xi_j \phi_{j,N},
\]
where the first approximation  corresponds to using the optimal projection subspace \(V_d\)
while the second approximation replaces 
\(V_d\) with \(V_{d,N}\).
Since it holds that \(\E C_N = C\), we expect \(\lambda_j \approx \lambda_{j,N}\) and \(\phi_j \approx \phi_{j,N}\). These discussions suggest that the quality of the PCA approximation
is controlled, on average, by the rate of decay of the eigenvalues of \(C\), and 
the approximation of the eigenstructure of $C$ by that of $C_N.$

\subsection{Neural Networks}
\label{sec:neuralnets}

A neural network is a nonlinear function  \(\chi : \R^n \to \R\) defined by a sequence of compositions 
of affine maps with point-wise nonlinearities. In particular, 
\begin{equation}\label{nn-general-form}
\tchi(s) = W_\Lt \sigma ( \dots \sigma ( W_2 \sigma (W_1 s + b_1) + b_2)) + b_\Lt, \qquad s \in \R^n,
\end{equation}
where \(W_1,\dots,W_\Lt\) are {\it weight matrices} (that are not necessarily square)
  and \(b_1,\dots,b_\Lt\) are vectors,
  referred to as {\it biases}. We refer to $\Lt \ge 1$ as the {\it depth of the neural network}.
The  function \(\sigma : \R^d \to \R^d\) is a monotone, nonlinear 
{\it activation function} that is defined from a monoton function \(\sigma : \R \to \R\) applied entrywise to any vector in $\R^d$ with $d \ge 1$.
  Note that in \eqref{nn-general-form} the input dimension of $\sigma$ may vary between layers
  but regardless of the input dimension the function $\sigma$ applies the same operations to
  all entries of the input vector. 
  We primarily consider the Rectified Linear Unit (ReLU) activation functions, i.e.,
  \begin{equation}\label{ReLU-def}
    \sigma(s) := (\max\{ 0, s_1\}, \max\{0, s_2\}, \dots, \max\{0, s_d\})^T \in
    \mathbb R^d\qquad \forall s \in \mathbb R^d.
  \end{equation}

The weights and biases constitute the parameters of the 
network. In this paper we learn these parameters in the following
standard way \cite{lecunnature}: given a set of data 
$\{x_j,y_j\}_{j=1}^N$ we choose the parameters of $\tchi$ to 
solve an appropriate regression problem by minimizing a data-dependent 
cost functional, using stochastic gradient methods.
Neural networks have been  demonstrated to constitute an efficient 
class of regressors and interpolators for high-dimensional problems empirically, 
but a complete theory of their efficacy is elusive. For an overview of various neural network 
architectures and their applications, see \cite{deeplearningbook}. 
For theories concerning their approximation capabilities 
see \cite{pinkus,yarotsky,schwab2019deep,devoredeep,kutyniok2019theoretical}.

For the approximation results given in Section \ref{sec:analysis}, we will 
work with a specific class of neural networks, following \cite{yarotsky}; we
note that other approximation schemes could be used, however, and that
we have chosen a proof setting that aligns with, but is not identical
to, what we implement in the computations described in 
Section \ref{sec:numerics}.
We will fix \(\sigma \in C(\R;\R)\) to be the ReLU   
function \eqref{ReLU-def} and consider the set of  neural networks mapping $\R^n$ to $\R$
\begin{equation*}
  \mathcal{M}(n; \Lt, \Jr) := \left \{ 
  \begin{aligned}
    & \tchi(s) = W_\Lt \sigma ( \dots \sigma ( W_2 \sigma (W_1 s + b_1) + b_2)) + b_\Lt \in \R,
    \\
 &  \text{for all } s\in \R^n \text{ and such that }  \sum_{k =1}^\Lt | W_k|_0 + | b_k|_0 \le  \Jr.
\end{aligned}
 \right \}
\end{equation*}
Here $| \cdot |_0$ gives the number of non-zero entries in a matrix so that
$\Jr \ge 0$ denotes the number of active  weights and biases
in the network while  $\Lt \ge 1$ is the total  number of layers.
Moreover, we define the class of stacked neural networks mapping $\R^n$ to $\R^m$:
\begin{equation*}
  \mathcal{M}(n, m; \Lt, \Jr) := \left \{ 
  \begin{aligned}
    & \chi(s) = (\chi^{(1)}(s), \dots, \chi^{(m)}(s) )^T  \in \R^m,
    \\
    &   \text{where }\chi^{(j)} \in \mathcal{M}(n; \Lt^{(j)}, \Jr^{(j)}),
    \text{ with } \Lt^{(j)} \le t, \Jr^{(j)} \le r.
\end{aligned}
 \right \}
\end{equation*}
From this, we build the set of \textit{zero-extended} neural networks
\begin{equation*}
  \mathcal{M}(n, m; \Lt, \Jr,  M) := \left \{ 
    \chi =
    \left\{
      \begin{aligned}
\tilde{\chi}(s), \quad & s \in [-M,M]^{n} \\
0, \quad & s \not \in [-M,M]^{n} 
\end{aligned}
\right\}, \text{  for some  } \tilde{\chi} \in \mathcal{M}( n, m, \Lt, \Jr)
 \right \},
\end{equation*}
where the new parameter
\(M > 0\) is the side length of the hypercube in $\R^n$ within which $\chi$ can be non-zero.
This construction is essential to our approximation as it allows us to handle non-compactness of the
latent spaces after 
PCA dimension reduction.

\subsection{Comparison to Existing Methods}
\label{sec:representation}

In the general setting of arbitrary encoders, the formula \eqref{eq:approxpsi} for
the approximation of $\Psi$ yields a complicated map, the representation of 
which depends on the dimension reduction methods being employed. However,
in the setting where PCA is used, a clear representation emerges which we now 
elucidate in order to highlight similarities and differences between our
methodology and existing methods appearing in the literature.

Let \(\FX :\X \to \R^{\dX}\) be the PCA encoder w.r.t. the data \(\{x_j\}_{j=1}^N\)
given by \eqref{eq:encoder} and,  in particular, let \(\phi^\X_{1,N},\dots,\phi^\X_{\dX,N}\) 
be the eigenvectors of the resulting empirical covariance. Similarly
let \(\phi^\Y_{1,N},\dots,\phi^\Y_{\dY,N}\) be the eigenvectors of the empirical covariance 
w.r.t. the data \(\{y\}_{j=1}^N\). For the function $\varphi$ defined
in \eqref{eq:approxphi}, or similarly for approximations $\chi$ thereof found through the use 
of neural networks,
we denote the components by
\(\varphi(\s) = (\varphi_1(\s), \dots, \varphi_{\dY}(\s))\)
for any \(s \in \R^{\dX}\). Then \eqref{eq:approxpsi} becomes
$\Psi(x) \approx \sum_{j=1}^{\dY} \alpha_j(x) \phi^\Y_{j,N}$
with coefficients
\[\alpha_j(x) = \varphi_j \big ( \FX(x) \big ) = \varphi_j
  \big ( \langle x, \phi^\X_{1,N} \rangle_\X,\dots, \langle x, \phi^\X_{\dX,N} \rangle_\X \big ), \qquad
\forall x \in \X.\]
 The solution data \(\{y\}_{j=1}^N\) fixes a basis for the
output space, and the dependence of \(\Psi(x)\) on \(x\) is captured solely 
via the scalar-valued coefficients \(\alpha_j\). This parallels the formulation
of the classical reduced basis method \cite{DeVoreReducedBasis}
where the approximation is written as 
\begin{equation}
\label{eq:thisa}
\Psi(x) \approx \sum_{j=1}^{m} \alpha_j(x) \phi_j.
\end{equation}
Many versions of the method exist, but two particularly popular ones 
are: (i) when \(m = N\) and \(\phi_j = y_j\); and (ii) when, as is 
done here, \(m=\dY\) and \(\phi_j = \phi^\Y_{j,N}\). The latter choice 
is also referred to as the reduced basis with a proper orthogonal decomposition.

The crucial difference between our method and the RBM
is in the formation of the coefficients \(\alpha_j\). In RBM these functions
 are obtained in an intrusive manner by approximating the PDE within the finite-dimensional
reduced basis and as a consequence the method cannot be used in a setting where 
a PDE relating inputs and outputs is not known, or may not exist. 
 In contrast, our proposed 
methodology approximates \(\varphi\) by regressing or interpolating 
the latent representations
$\{\FX (x_j), \FY(y_j)\}_{j=1}^N$.
Thus our proposed method makes use of the entire available dataset and does not require
explicit knowledge of the underlying PDE mapping, making it a non-intrusive method applicable
to black-box models.

The form \eqref{eq:thisa} of the approximate solution operator can also be related 
to the Taylor approximations developed in \cite{cohenalgo,cohenconv} where
a particular form of the input $x$ is considered, namely 
$x = \bar{x} + \sum_{j \geq 1} a_j \tx_j$
where \(\bar{x} \in \X\) is fixed, \(\{a_j\}_{j \geq 1} \in \ell^\infty(\mathbb{N};\R)\) 
are uniformly bounded, and \(\{\tx_j\}_{j \geq 1} \in \X\)
have some appropriate norm decay. Then, assuming that
the solution operator \(\Psi : \X \to \Y\) is 
analytic \cite{cohenanalytic}, it is possible to make use
of the Taylor expansion
\[\Psi(x) = \sum_{h \in \mathcal{F}} \alpha_h(x) \psi_h, \]
where \(\mathcal{F} = \{h \in \mathbb{N}^\infty : |h|_0 < \infty\}\) 
is the set of multi-indices and
\[\alpha_h(x) = \prod_{j \geq 1} a_j^{h_j} \in \R,
  \qquad \psi_h = \frac{1}{h!} \partial^h \Psi(0) \in \Y;\]
here the differentiation \(\partial^h\) is with respect to the sequence of coefficients \(\{a_j\}_{j \geq 1}\).  
Then $\Psi$ is  approximated  by truncating the Taylor expansion to a finite subset of \(\mathcal{F}\). 
For example this may be done recursively, by starting with \(h = 0\)
and building up the index set in a greedy manner.
The method is not data-driven, and requires knowledge of the PDE to
define equations to be solved for the $\psi_h$.


\section{Approximation Theory}
\label{sec:analysis}

In this section, we prove our main approximation result:
given any $\eps>0$, we can find an $\eps-$approximation $\Psin$ of $\Psi$.
We achieve this by making the appropriate choice of PCA truncation parameters, by choosing
sufficient amounts of data, and by choosing a sufficiently rich neural
network architecture to approximate $\varphi$ by $\chi.$  

In what follows we define
\(\FX\) to be a PCA encoder given by \eqref{eq:encoder}, using
the input data $\{ x_j\}_{j=1}^N$ drawn i.i.d. from $\mu$,
and \(\GY\) to be a PCA decoder given by
\eqref{eq:decoder}, using the data  $\{ y_j = \Psi(x_j) \}_{j=1}^N$.
We also define
\begin{align*}
\en(x) &= \|(\GY \circ \chi \circ \FX)(x) - \Psi(x)\|_Y\\
&=\|\Psin(x) - \Psi(x)\|_\Y.
\end{align*}
We prove the following theorem:

\begin{theorem}
\label{thm:limit}
Let \(\X\), \(\Y\) be real, separable Hilbert spaces and let \(\mu\) be a probability measure 
supported on \(\X\) such that \(\E_{x \sim \mu} \|x\|_\X^4 < \infty\). Suppose \(\Psi : \X \to \Y\) is a
$\mu$-measurable, globally Lipschitz map.
For any \(\epsilon > 0\), there are dimensions \(\dX = \dX(\epsilon) \in \mathbb{N}\),
\(\dY = \dY(\epsilon) \in \mathbb{N}\), a requisite amount of data
\(N = N(\dX, \dY) \in \mathbb{N}\), parameters $\Lt, \Jr, M$
depending on $\dX,\dY$ and $\epsilon$, and a zero-extended  
stacked neural network $\chi \in \mathcal{M}(\dX, \dY; \Lt, \Jr, M)$ such that
\[\E_{\{x_j\} \sim \mu} \E_{x \sim \mu}\bigl(\en(x)^2\bigr) < \epsilon.\]
\end{theorem}

\begin{remark}
This theorem is a consequence of Theorem \ref{thm:approximation}
which we state and prove below.
For clarity and ease of exposition we state and prove
Theorem \ref{thm:approximation} in a setting where \(\Psi\) is
globally Lipschitz. With a more stringent moment condition
on \(\mu\), the result can also be proven when $\Psi$
is locally Lipschitz; we state and prove
this result in Theorem \ref{thm:approximation-local-lip}.
\end{remark}

\begin{remark}
The neural network
\(\chi \in \mathcal{M}(\dX, \dY; \Lt, \Jr, M)\)
has maximum number of layers   $t \le c [ \log (M^2 \dY/ \epsilon) +1 ]$,
with the number of active weights and biases in each component of the network 
$r \le c (\epsilon/4M^2)^{-\dX/2}[ \log (  M^2 \dY/\epsilon) + 1]$, 
with an appropriate constant $c = c( \dX, \dY) \ge 0$ 
and support side-length \(M = M(\dX,\dY) > 0\).
These bounds on $t$ and $r$ follow from Theorem \ref{thm:approximation} 
with $\tau=\epsilon^{\frac12}.$
Note, however, that in order to achieve error $\epsilon$, the
dimensions $\dX, \dY$ must be chosen to grow as $\epsilon \to 0$; 
thus the preceding statements
do not explicitly quantify the needed number of parameters, and depth,
for error $\epsilon$; to do so would require quantifying the dependence
of $c,M$ on $\dX,\dY$ (a property of neural networks) and the dependence
of  $\dX,\dY$ on $\epsilon$ (a property of the measure $\mu$ and spaces
$\X,\Y$ -- see Theorem \ref{thm:pca_generalization_bound}).
The theory in \cite{yarotsky}, which we 
employ for the existence result for the neural network
produces the constant \(c\) which depends on the dimensions 
\(\dX\) and \(\dY\) in an unspecified way. 
\end{remark}


The double expectation reflects averaging over all possible 
new inputs $x$ drawn from $\mu$ (inner expectation) and over 
all possible realizations of the i.i.d. dataset $\{ x_j, y_j = \Psi(x_j) \}_{j=1}^N$ (outer expectation).
The theorem as stated above is a consequence of Theorem \ref{thm:approximation} 
in which the error is broken into multiple components that are then bounded separately.
Note that the theorem does not address the question of whether the
optimization technique used to fit the neural network actually finds the
choice which realizes the theorem; this gap between theory and
practice is difficult to overcome, because of the non-convex nature of
the training problem, and is a standard feature of theorems in this
area \cite{kutyniok2019theoretical,schwab2019deep}.

The idea of the proof is to quantify the approximations \(\GX \circ \FX \approx I_\X\) and 
\(\GY \circ \FY \approx I_\Y\) and $\chi \approx \varphi$ so that $\Psin$
given by \eqref{eq:apsipcn} is close to $\Psi.$ 
The first two approximations, which show that $\Psip$ given by
\eqref{eq:apsipca} is close to $\Psi,$ are studied in Subsection \ref{sec:pcaanalysis} 
(see Theorem \ref{thm:pca_generalization_bound}). 
Then, in Subsection \ref{sec:approxanalysis}, we find a neural network 
\(\chi\) able to approximate \(\varphi\) to the desired level of accuracy;
this fact is part of the proof of Theorem \ref{thm:approximation}. 
The zero-extension of the neural network arises from the fact that 
we employ a density theorem for a class of neural networks within 
continuous functions defined on compact sets.  Since we cannot 
guarantee that \(\FX\) is bounded, we simply set the neural network output to zero 
on the set outside a hypercube with side-length \(2M\). We then use the fact that 
this set has small $\mu$-measure, for  sufficiently large $M.$

\subsection{PCA And Approximation}
\label{sec:pcaanalysis}
We work in the general notation and setting of Subsection \ref{sec:functionalpca}
so as to obtain approximation results  that are applicable to both using PCA on the 
inputs and on the outputs. 
In addition, denote by \((\HS, \langle \cdot, \cdot \rangle_{HS}, \|\cdot\|_{HS})\) 
the space of Hilbert-Schmidt operators over \(\mathcal{H}\).  We are now ready
to state the main result of this subsection.  Our goal is to control 
the projection error \(R(V_{d,N})\) when using the finite-data PCA subspace in place of the 
optimal projection space since the PCA subspace is 
what is available in practice. Theorem \ref{thm:pca_generalization_bound} accomplishes this by bounding the error \(R(V_{d,N})\) by the optimal error \(R(V_d)\) plus a 
term related to the approximation \(V_{d,N} \approx V_d\).
While previous results such as \cite{Blanchard2007,ShaweTaylor2,ShaweTaylor1} 
focused on bounds for the excess error in probability w.r.t. the data, 
we present bounds in expectation, averaging over the data. 
Such bounds are weaker, but allow us to remove strict conditions on the data 
distribution to obtain more general results; for example, our theory 
allows for \(\nu\) to be a Gaussian measure.

\begin{theorem}
\label{thm:pca_generalization_bound}
Let \(R\) be given by \eqref{eq:projectionerror} and \(V_{d,N}\), \(V_d\) by \eqref{eq:pcasubspace}, \eqref{eq:optprojectionspace} respectively. Then there 
exists a constant \(Q \geq 0\), depending only on the data generating measure \(\nu\), such that
\[\E_{\{u_j\} \sim \nu} [R(V_{d,N})] \leq \sqrt{\frac{Qd}{N}} + R(V_d),\]
where the expectation is over the dataset $\{ u_j \}_{j=1}^N \stackrel{iid}{\sim} \nu$.
\end{theorem}

The proof generalizes that employed in \cite[Thm. 3.1]{Blanchard2007}.
We first find a bound on the average excess error 
\(\E [R(V_{d,N}) - R_N (V_{d,N})]\) using Lemma \ref{lemma:convariancemontecarlo}.
Then using Fan's Theorem \cite{Fan} (Lemma \ref{thm:fan}), 
we bound the average sum of the tail eigenvalues of \(C_N\) by the sum of the 
tail eigenvalues of \(C\), in particular, \(\E [R_N (V_{d,N})] \leq R(V_d) \).

\begin{proof}
  For brevity we simply write $\E$ instead of $\E_{\{u_j\} \sim \nu}$ throughout the proof.
For any subspace \(V \subseteq \mathcal{H}\), we have
\begin{align*}
R(V) &= \E_{u \sim \nu} [\|u\|^2 -  2 \langle u, \Pi_V u \rangle + \langle \Pi_V u, \Pi_V u \rangle]= \E_{u \sim \nu} [\tr (u \otimes u) -  \langle \Pi_V u , \Pi_V u \rangle ] \\
&= \E_{u \sim \nu} [\tr (u \otimes u)- \langle \Pi_V, u \otimes u \rangle_{HS}] = \tr C - \langle \Pi_V, C \rangle_{HS}
\end{align*}
where we used two properties of the fact that \(\Pi_V\) is an 
orthogonal projection operator, namely \(\Pi_V^2 = \Pi_V = \Pi_V^*\) and
\[\langle \Pi_V, v \otimes z \rangle_{HS} = \langle v, \Pi_V z \rangle = \langle \Pi_V v, \Pi_V z \rangle
  \quad \forall v,z \in \mathcal{H}.\]
Repeating the above arguments for $R_N(V)$ in place of $R(V)$, with the expectation replaced
by the empirical average, yields
$R_N(V) = \tr C_N - \langle \Pi_V, C_N \rangle_{HS}.$
By noting that \(\E[C_N] = C\) we then write
\begin{align*}
  \E [R(V_{d,N}) - R_N(V_{d,N})]
 &  = \E \langle \Pi_{V_{d,N}}, C_N - C \rangle_{HS}
  \leq \sqrt{d} \: \E \|C_N - C\|_{HS}  \\
  & \leq \sqrt{d} \sqrt{\E \|C_N - C\|^2_{HS}}
\end{align*}
where we used
 Cauchy-Schwarz twice along with the fact that 
\(\|\Pi_{V_{d,N}}\|_{HS} = \sqrt{d}\) since
\(V_{d,N}\) is \(d\)-dimensional.
Now by Lemma \ref{lemma:convariancemontecarlo}, which quantifies the Monte Carlo
error between $C$ and $C_N$ in the Hilbert-Schmidt norm, we have that
\[\E [R(V_{d,N}) - R_N(V_{d,N})] \leq   \sqrt{\frac{Qd}{N}},\]
for   
a constant \(Q \geq 0\).
Hence by\eqref{eq:nt}, 
\begin{align*}
\E [R(V_{d,N})]&\leq  \sqrt{\frac{Qd}{N}} + \E \sum_{j=d+1}^N \lambda_{j,N}.
\end{align*}
It remains to estimate the second term  above.
Letting \(S_d\) denote the set of subspaces of \(d\) orthonormal elements in \(\mathcal{H}\),
 Fan's Theorem (Proposition~\ref{thm:fan}) gives
\begin{align*}
  \sum_{j=1}^d \lambda_j
  &= \max_{\{v_1,\dots,v_d\} \in S_d} \sum_{j=1}^d \langle Cv_j, v_j \rangle
   = \max_{\{v_1,\dots,v_d\} \in S_d} \E_{u \sim \nu} \sum_{j=1}^d |\langle u, v_j \rangle|^2 \\
&= \max_{\{v_1,\dots,v_d\} \in S_d} \E_{u \sim \nu} \sum_{j=1}^d \|\Pi_{\text{span}\{v_j\}} u\|^2 
= \max_{V \in \mathcal{V}_d} \E_{u \sim \nu} \|\Pi_V u\|^2 \\
&= \E_{u \sim \nu} \|u\|^2 - \min_{V \in \mathcal{V}_d} E_{u \sim \nu} \|\Pi_{V^\perp} u\|^2.
\end{align*}
Observe that $\sum_{j=1}^\infty \lambda_j = \tr C  = \E_{u \sim \nu} \|u\|^2$
and so 
\[\sum_{j=d+1}^\infty \lambda_j = \E_{u \sim \nu} \|u\|^2  - \sum_{j=1}^d \lambda_j
  = \min_{V \in \mathcal{V}_d} \E_{u \sim \nu} \|\Pi_{V^\perp} u\|^2.\]
We now repeat the above calculations  for $\lambda_{j,N}$, the eigenvalues of $C_N$, by replacing
the expectation with the empirical average to obtain
\[\sum_{j=d+1}^N \lambda_{j,N} = \min_{V \in \mathcal{V}_d} \frac{1}{N} \sum_{k=1}^N \|\Pi_{V^\perp} u_k\|^2,\]
and so  
\begin{align*}
\E \sum_{j=d+1}^N \lambda_{j,N} &\leq \min_{V \in \mathcal{V}_d} \E \frac{1}{N} \sum_{k=1}^N \|\Pi_{V^\perp} u_k\|^2 
= \min_{V \in \mathcal{V}_d} \E_{u \sim \mu} \|\Pi_{V^\perp} u\|^2 
= \sum_{j=d+1}^\infty \lambda_j.
\end{align*}
Finally, we conclude that
\begin{align*}
\E [R(V_{d,N})] \leq  \sqrt{\frac{Qd}{N}} + \sum_{j=d+1}^\infty \lambda_j
= \sqrt{\frac{Qd}{N}} + R(V_d).
\end{align*}

\end{proof}

\subsection{Neural Networks And Approximation}
\label{sec:approxanalysis}
In this subsection we study the approximation of $\varphi$ given in
\eqref{eq:approxphi} by neural networks,
combining the analysis with results from the preceding subsection to prove
our main approximation result, Theorem \ref{thm:approximation}.
We will work in the notation of Section \ref{sec:method}. 
We assume that \((\X, \langle \cdot, \cdot \rangle_\X, \|\cdot\|_\X) \) and \((\Y, \langle \cdot, \cdot \rangle_\Y, \|\cdot\|_\Y) \) 
are real, separable Hilbert spaces;
\(\mu\) is a probability measure supported on \(\X\) with a finite fourth moment 
\(\E_{x \sim \mu} \|x\|_\X^4 < \infty\), and \(\Psi: \X \to \Y\) is measurable and globally \(L\)-Lipschitz:
there exists a constant $L>0$ such that 
\[\forall x,z \in \X\quad \|\Psi(x) - \Psi(z)\|_\Y \leq L \|x-z\|_\X.\]
Note that this implies that \(\Psi\) is linearly bounded: for any \(x \in \X\)
\[\|\Psi(x)\|_\Y \le  \|\Psi(0)\|_\Y + \|\Psi(x) - \Psi(0)\|_\Y \leq \|\Psi(0)\|_\Y + L\|x\|_\X.\]
Hence we deduce existence of the fourth moment of the pushforward \(\Psi_\sharp \mu\):
\[
\E_{y \sim \Psi_\sharp \mu} \|y\|^4_\Y = \int_\X \|\Psi(x)\|_\Y^4 d\mu(x) \leq \int_\X (\|\Psi(0)\|_\Y + L \|x\|_\X)^4 d\mu(x) < \infty
\]
since we assumed \(\E_{x \sim \mu} \|x\|^4_\X < \infty\).

Let us recall some of the notation from Subsections \ref{sec:functionalpca} and \ref{sec:representation}.
Let \(V^\X_{\dX}\) be the \(\dX\)-dimensional optimal projection space
given by \eqref{eq:optprojectionspace} for the measure \(\mu\) and \(V^\X_{\dX,N}\) be the \(\dX\)-dimensional
PCA subspace given by \eqref{eq:pcasubspace} with respect to the input dataset \(\{x_j\}_{j=1}^N\). Similarly 
let \(V^\Y_{\dY}\) be the \(\dY\)-dimensional optimal projection space
for the pushforward measure \(\Psi_\sharp \mu\) and \(V^\Y_{\dY,N}\) be the \(\dY\)-dimensional
PCA subspace with respect to the output dataset \(\{y_j = \Psi(x_j)\}_{j=1}^N\). We then define the input
PCA encoder
\(\FX : \X \to \R^{\dX}\) by \eqref{eq:encoder} and the input PCA decoder \(\GX: \R^{\dX} \to \X\)
by \eqref{eq:decoder} both with respect to the orthonormal basis used to construct \(V^\X_{\dX,N}\).
Similarly we define the output PCA encoder \(\FY : \Y \to \R^{\dY}\) and decoder \(\GY : \R^{\dY} \to \Y\)
with respect to the orthonormal basis used to construct \(V^\Y_{\dY,N}\). Finally
we recall \(\varphi : \R^{\dX} \to \R^{\dY}\) the map connecting the two latent 
spaces  defined in \eqref{eq:approxphi}. 
The approximation \(\Psip\) to \(\Psi\) based only on the PCA encoding and decoding
is given by \eqref{eq:apsipca}. 
In the following theorem, we prove the existence of a neural network giving an 
\(\epsilon\)-close approximation to \(\varphi\) for fixed latent code dimensions $\dX, \dY$
and quantify the error of 
the full approximation \(\Psin\), given in \eqref{eq:apsipcn}, to \(\Psi\).
We will be explicit  about which measure the projection error is defined with 
respect to. In particular, we will write \eqref{eq:projectionerror} as 
\[R^\mu(V) = \E_{x \sim \mu} \|x - \Pi_V x\|^2_\X\]
for any subspace \(V \subseteq \X\) and similarly
\[R^{\Psi_\sharp \mu}(V) = \E_{y \sim \Psi_\sharp \mu} \|y - \Pi_V y\|^2_\Y\]
for any subspace \(V \subseteq \Y\).

\begin{theorem}
\label{thm:approximation}
Let \(\X\), \(\Y\) be real, separable Hilbert spaces and let \(\mu\) be a probability measure
supported on \(\X\) such that \(\E_{x \sim \mu} \|x\|_\X^4 < \infty\). Suppose \(\Psi : \X \to \Y\) is a
$\mu$-measurable, globally Lipschitz map.
Fix $\dX, \dY$, $N \ge \max\{\dX, \dY\},$ $\delta \in (0,1)$ and $\tau>0$.
Define \(M = \sqrt{\E_{x \sim \mu} \|x\|_\X^2 / \delta}\). Then 
there exists a constant $c = c(\dX, \dY) \ge 0$ and a zero-extended stacked neural network
\(\chi \in \mathcal{M}(\dX, \dY; t, r,M)\)
with $t \le c(\dX,\dY) [  \log ( M \sqrt{\dY}/ \tau) +1 ]$ and
$r \le c(\dX,\dY)  (\tau/2M)^{-\dX}[ \log (M \sqrt{\dY}/\tau) + 1]$, 
so that 
\begin{equation}
\label{eq:approximationbound}
\begin{aligned}
\E_{\{x_j\} \sim \mu} \E_{x \sim \mu}\bigl(\en(x)\bigr)^2 \le C \Bigg(  \tau^2 + \sqrt{\delta}
  + \sqrt{\frac{\dX}{N}} + R^{\mu}(V_{\dX}^{\X}) +  \sqrt{\frac{\dY}{N}} + R^{\Psi_\sharp \mu}(V_{\dY}^{\Y}) \Bigg),
\end{aligned}
\end{equation}
where 
$C > 0$ is independent of $\dX, \dY, N, \delta$ and $\tau$.
\end{theorem}

The first two terms on the r.h.s. arise from the neural network approximation of \(\varphi\) while the 
last two pairs of terms are from the finite-dimensional approximation of \(\X\) and \(\Y\) respectively 
as prescribed by Theorem \ref{thm:pca_generalization_bound}.
The way to interpret the result is as follows: first choose $\dX, \dY$ so that
$R^{\mu}(V_{\dX}^{\X})$ and $R^{\Psi_\sharp \mu}(V_{\dY}^{\Y})$ are 
small -- these are intrinsic properties of the measures $\mu$ and $\Psi_\sharp \mu$; 
secondly, choose the amount of data $N$ large enough to make 
$\max\{\dX, \dY\}/N$ small, essentially controlling how well we
approximate the intrinsic covariance structure of $\mu$ and $\Psi_\sharp \mu$ using samples; 
thirdly choose $\delta$ small enough to control the error arising from restricting the domain of 
$\varphi$; and finally choose $\tau$ sufficiently small to control the 
approximation of $\varphi$ by a neural network restricted to a compact set.
Note that the size and values of the parameters of the neural network \(\chi\) will depend on the choice of
 \(\delta\) as well as $\dX, \dY$ and $N$ in a manner which we do not specify.
In particular, the dependence of  $c$ on $\dX, \dY$ is not
explicit in the theorem of \cite{yarotsky} which furnishes the
existence of the requisite neural network $\chi.$
The parameter \(\tau\) specifies the
error tolerance between \(\chi\) and \(\varphi\) on \([-M,M]^{\dX}\). 
Intuitively, as \((\delta,\tau) \to 0\), we expect  the number of 
parameters in the network to also grow \cite{pinkus}. Quantifying this growth
would be needed to fully understand the computational complexity of our method.

\begin{proof}
Recall the constant $Q$ from Theorem \ref{thm:pca_generalization_bound}. In
what follows we take $Q$ to be the maximum of the two such constants
when arising from application of the theorem on the two different probability
spaces $(\X,\mu)$ and $(\Y, \Psi_\sharp \mu).$ Through the proof we use $\E$
to denote $\E_{\{x_j \} \sim \mu}$ the expectation with respect to the dataset $\{ x_j\}_{j=1}^N$.

We begin by approximating the error incurred by using \(\Psip\) given by
\eqref{eq:apsipca}:
\begin{align}
\label{eq:truephiapprox}
\begin{split}
  \E  & \E_{x \sim \mu} \|\Psip(x) - \Psi(x)\|^2_\Y \\
  &= \E \E_{x \sim \mu} \|(\GY \circ \FY \circ \Psi \circ \GX \circ \FX)(x) - \Psi(x)\|^2_\Y \\
&= \E \E_{x \sim \mu} \left\|\Pi_{V_{\dY,N}^\Y} \Psi (\Pi_{V_{\dX,N}^\X} x) - \Psi(x) \right\|^2_\Y \\
&\leq 2 \E \E_{x \sim \mu} \left\|\Pi_{V_{\dY,N}^\Y} \Psi(\Pi_{V_{\dX,N}^\X} x) - \Pi_{V_{\dY,N}^\Y} \Psi(x) \right\|^2_\Y \\ 
&\:\:\:\:+ 2 \E \E_{x \sim \mu}\left \|\Pi_{V_{\dY,N}^\Y} \Psi(x) - \Psi(x) \right\|^2_\Y \\
& \leq 2 L^2 \E \E_{x \sim \mu} \left\|\Pi_{V_{\dX,N}^\X} x - x \right\|^2_\X
+ 2 \E \E_{y \sim \Psi_\sharp \mu} \left\|\Pi_{V_{\dY,N}^\Y} y - y \right\|^2_\Y \\
&= 2 L^2 \E [R^{\mu}(V_{\dX,N}^\X)] + 2 \E [R^{\Psi_\sharp \mu}(V_{\dY,N}^\Y)]
\end{split}
\end{align}
noting that the operator norm of an orthogonal projection is \(1\). 
Theorem \ref{thm:pca_generalization_bound}
allows us to control this error, and leads to 
\begin{align}
\label{eq:truephiapprox2}
\begin{split}
\E \E_{x \sim \mu} \bigl(\en(x)^2\bigr) &=\E \E_{x \sim \mu} \|\Psin(x) - \Psi(x)\|^2_\Y\\
& \le 2\E \E_{x \sim \mu}\|\Psin(x) - \Psip(x)\|^2_\Y+2\E \E_{x \sim \mu} \|\Psip(x) - \Psi(x)\|^2_\Y\\
& \le 2 \E_{x \sim \mu}\|\Psin(x) - \Psip(x)\|^2_\Y\\
&\quad\quad+4L^2\left( \sqrt{\frac{Q \dX}{N}} + R^{\mu}(V_{\dX}^\X) \right)+
4\left(\sqrt{\frac{Q \dY}{N}} + R^{\Psi_\sharp \mu}(V_{\dY}^\Y)\right).
\end{split}
\end{align}

We now approximate \(\varphi\) by a neural network \(\chi\) as a step
towards estimating $\|\Psin(x) - \Psip(x)\|_\Y.$ 
To that end we first note from Lemma \ref{l:add} that 
\(\varphi\) is Lipschitz, and hence continuous, as a mapping from $\R^{\dX}$
into $\R^{\dY}.$ 
Identify the components \(\varphi(s) = (\varphi^{(1)}(s), \dots, \varphi^{(\dY)}(s))\)
where each function \(\varphi^{(j)} \in C(\R^{\dX};\R)\).
We consider the restriction of each component function to the set  \([-M,M]^{\dX}\).
Let us now change variables by defining \(\tilde{\varphi}^{(j)} : [0,1]^{\dX} \to \R\) by
\(\tilde{\varphi}^{(j)}(s) = (1/2M) \varphi^{(j)}(2Ms - M)\) for any \(s \in [0,1]^{\dX}\).
Note that   equivalently  we have \(\varphi^{(j)}(s) = 2M \tilde{\varphi}^{(j)}((s+M)/2M)\) for any \(s \in [-M,M]^{\dX}\) 
and further  \(\varphi^{(j)}\) and \(\tilde{\varphi}^{(j)}\) have the same
Lipschitz constants on their respective domains.
Applying \cite[Thm.~1]{yarotsky}  to the $\tilde{\varphi}^{(j)}$(s)  then yields  existence of neural networks
  \(\tilde{\chi}^{(1)}, \dots, \tilde{\chi}^{(\dY)} : [0, 1]^{\dX} \to \R\)
such that
\[|\tilde{\chi}^{(j)}(s) - \tilde{\varphi}^{(j)}(s)| < \frac{\tau}{2 M \sqrt{\dY}}
\quad \forall \s \in [0,1]^{\dX},\]
for any \(j \in \{1,\dots, \dY\}\). 
In fact, each neural network
$\tilde{\chi}^{(j)} \in \mathcal{M}(\dX; t^{(j)}, r^{(j)})$ with parameters $t^{(j)}$ and $r^{(j)}$ satisfying
\begin{equation*}
  t^{(j)} \le c^{(j)} \left[ \log( M \sqrt{\dY} / \tau) + 1 \right], \qquad
  r^{(j)} \le c^{(j)} \left(\frac{\tau}{2M}\right)^{- \dX} \left[ \log( M \sqrt{\dY} /\tau ) +1  \right], 
\end{equation*}
with constants $c^{(j)}(\dX) >0$.
Hence defining $\chi^{(j)}: \R^{\dX} \to \R$ by \(\chi^{(j)}(s) := 2M \tilde{\chi}^{(j)}((s+M)/2M)$
for any \(s \in [-M,M]^{\dX}\), we have that 
\begin{equation*}
  \big| \big( \chi^{(1)}(s), \dots, \chi^{(\dY)}(s) \big)  - \varphi(s) \big|_2 < \tau
  \quad \forall s \in [-M,M]^{\dX}.
\end{equation*}

We can now simply 
define \(\chi: \R^{\dX} \to \R^{\dY}\) as the  stacked network
\((\chi^{(1)},\dots,\chi^{\dY})\) extended by zero outside of $[-M, M]^{\dX}$
to  immediately obtain
\begin{equation}
\label{eq:nnepsilonclose}
\sup_{ s \in [-M,M]^{\dX}}
\big| \chi(s)  - \varphi(s) \big|_2 < \tau. 
\end{equation}
Thus, by construction 
$\chi \in \mathcal{M}(\dX, \dY, t, r, M)$ with at most
$t \le \max_j t^{(j)}$ many layers and
$r \le r^{(j)}$ 
many active weights and biases in each of its components.

Let us now define the set 
$A = \{x \in \X : \FX(x) \in [-M,M]^{\dX}\}.$
By Lemma \ref{lemma:bounded_projection}, \(\mu(A) \geq 1 - \delta\) and
\(\mu(A^c) \leq \delta\).  Define the approximation error 
\[\ep(x) = \|\Psin(x) - \Psip(x)\|_\Y\]
and decompose its expectation as 
\[\E_{x \sim \mu}\bigl(\ep(x)^2\bigr) = \underbrace{\int_A \ep(x)^2 d\mu(x)}_{\coloneqq I_A} + \underbrace{\int_{A^c} \ep(x)^2 d\mu(x)}_{\coloneqq I_{A^c}}.\]
For the first term,
\begin{align}
\label{eq:erroronA}
\begin{split}
  I_A &\leq  \int_A  \|(\GY \circ \chi \circ \FX)(x)
  - (\GY \circ \varphi \circ \FX)(x)\|^2_\Y d\mu(x)\leq \tau^2, 
\end{split}
\end{align}
by using the fact, established in Lemma \ref{l:add}, that
\(\GY\) is Lipschitz with Lipschitz constant $1$, the
\(\tau\)-closeness of \(\chi\) 
to \(\varphi\) from \eqref{eq:nnepsilonclose}, and \(\mu(A) \leq 1\). 
For the second term we have, using that
$\GY$ has Lipschitz constant $1$ and that $\chi$ vanishes on $A^c$, 
\begin{align}
\label{eq:erroroutsideA}
\begin{split}
I_{A^c} &\leq  \int_{A^c}  \|(\GY \circ \chi \circ \FX)(x) - (\GY \circ \varphi \circ \FX)(x)\|^2_\Y d\mu(x) \\ 
&\leq  \int_{A^c} | \chi(\FX(x)) - \varphi(\FX(x))|^2_2 d\mu(x)
=\int_{A^c} | \varphi(\FX(x))|^2_2 d\mu(x).
\end{split}
\end{align}
Once more from Lemma \ref{l:add}, we have that

\[|\FX(x)|_2 \leq \|x\|_{\X}; \quad |\varphi(x)|_2 \leq |\varphi(0)|_2 + L|x|_2,\]
so that 
\begin{align}
\label{eq:erroroutsideA2}
\begin{split}
I_{A^c} & \leq 2\bigl(\mu(A^c)|\varphi(0)|_2^2+\mu(A_c)^{\frac12}L^2  (\E_{x \sim \mu } \|x\|_\X^4)^{\frac12}\bigr),\\
& \leq 2 \bigl(\delta |\varphi(0)|_2^2+\delta^{\frac12}L^2  (\E_{x \sim \mu } \|x\|_\X^4)^{\frac12}\bigr). 
\end{split}
\end{align}
Combining \eqref{eq:truephiapprox2}, \eqref{eq:erroronA} 
and \eqref{eq:erroroutsideA2}, we obtain the desired result.
\end{proof}

\section{Numerical Results}
\label{sec:numerics}

\begin{figure}[t]
    \centering
    \begin{subfigure}[b]{0.24\textwidth}
        \includegraphics[width=\textwidth]{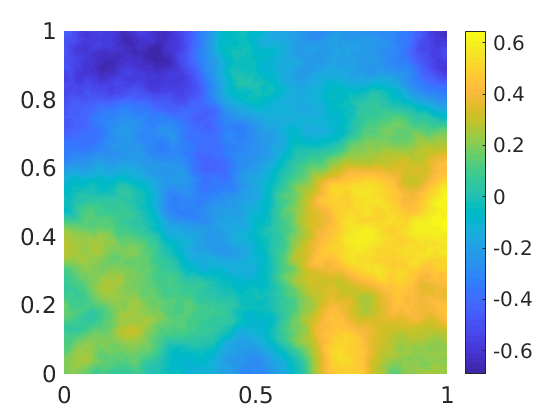}
        \caption{\(\mug\)}
    \end{subfigure}
    \begin{subfigure}[b]{0.24\textwidth}
        \includegraphics[width=\textwidth]{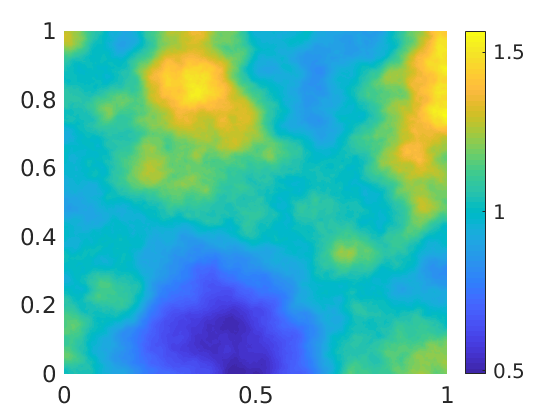}
        \caption{\(\mul\)}
    \end{subfigure}
    \begin{subfigure}[b]{0.24\textwidth}
        \includegraphics[width=\textwidth]{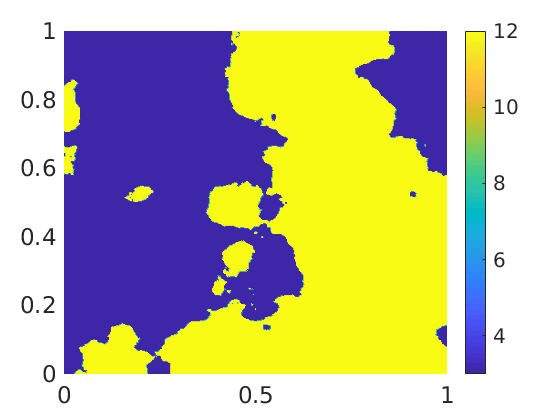}
        \caption{\(\mup\)}
    \end{subfigure}
    \begin{subfigure}[b]{0.24\textwidth}
        \includegraphics[width=\textwidth]{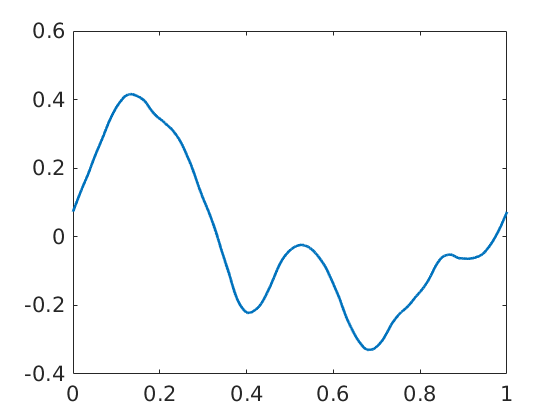}
        \caption{\(\mu_{\text{B}}\)}
    \end{subfigure}

    \caption{Representative samples for each of the probability measures $\mug, \mul, \mup, \mu_{\text{B}}$
      defined in Subsection \ref{ssec:PDES}. \(\mug\) and \(\mup\) are used in Subsection~\ref{sec:numlip}
      to model the inputs, \(\mul\) and \(\mup\) are used in Subsection~\ref{sec:numdarcy}, and $\mu_{\text{B}}$
      is used in Subsection~\ref{sec:burgers}.}
    \label{fig:samples} 
\end{figure}

\begin{figure}[h!]
    \centering
    \begin{subfigure}[b]{0.24\textwidth}
        \includegraphics[width=\textwidth]{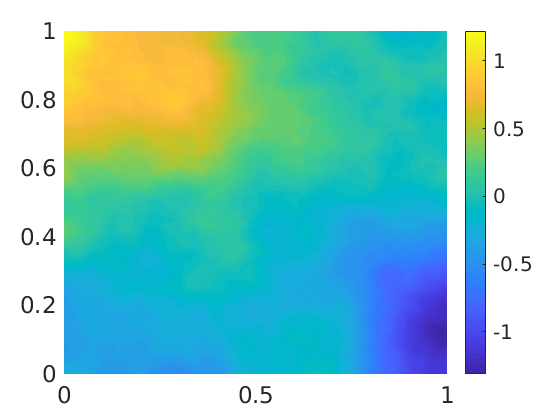}
        \caption{Input}
    \end{subfigure}
    \begin{subfigure}[b]{0.24\textwidth}
        \includegraphics[width=\textwidth]{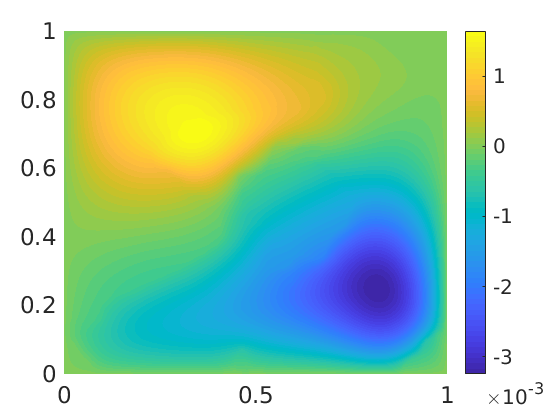}
        \caption{Ground Truth}
    \end{subfigure}
    \begin{subfigure}[b]{0.24\textwidth}
        \includegraphics[width=\textwidth]{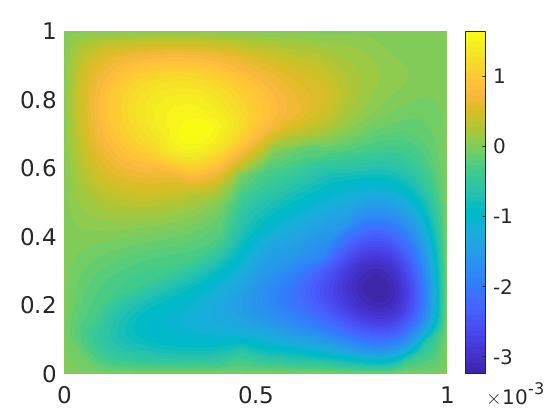}
        \caption{Approximation}
    \end{subfigure}
    \begin{subfigure}[b]{0.24\textwidth}
        \includegraphics[width=\textwidth]{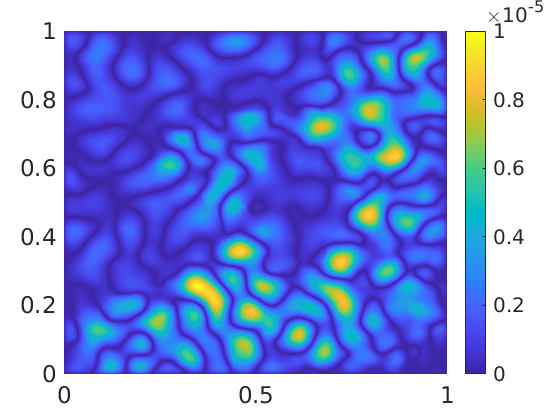}
        \caption{Error}
    \end{subfigure}

    \begin{subfigure}[b]{0.24\textwidth}
    \includegraphics[width=\textwidth]{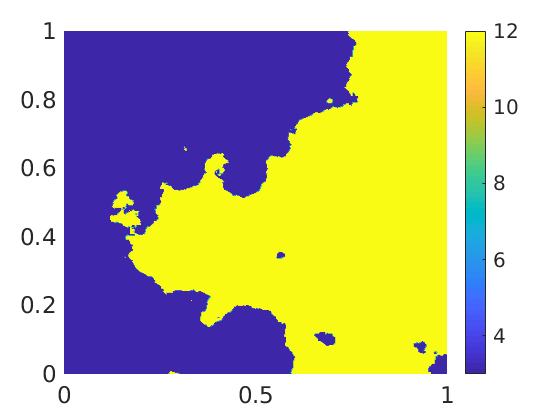}
    \caption{Input}
    \end{subfigure}
    \begin{subfigure}[b]{0.24\textwidth}
        \includegraphics[width=\textwidth]{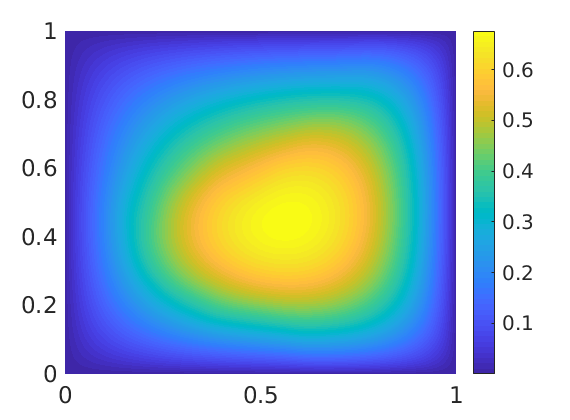}
        \caption{Ground Truth}
    \end{subfigure}
    \begin{subfigure}[b]{0.24\textwidth}
        \includegraphics[width=\textwidth]{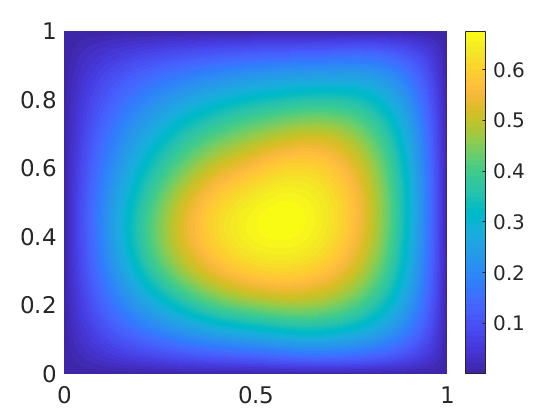}
        \caption{Approximation}
    \end{subfigure}
    \begin{subfigure}[b]{0.24\textwidth}
        \includegraphics[width=\textwidth]{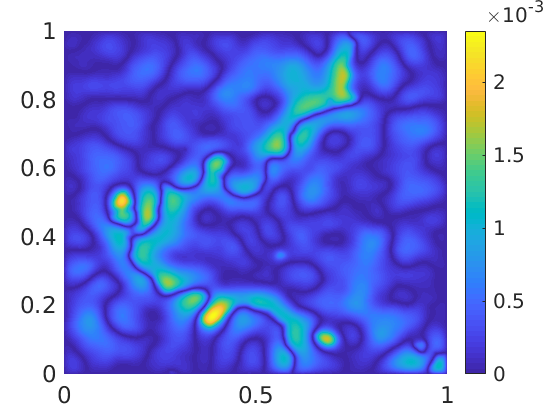}
        \caption{Error}
    \end{subfigure}

    \begin{subfigure}[b]{0.24\textwidth}
    \includegraphics[width=\textwidth]{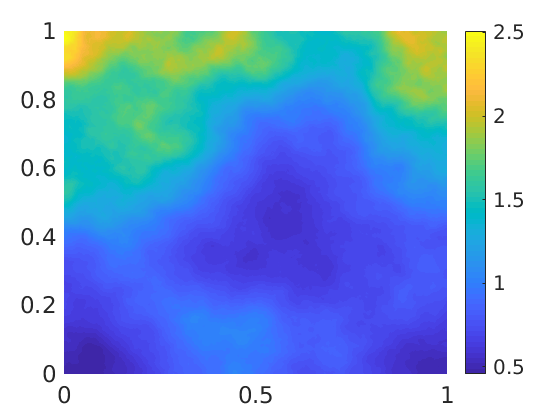}
    \caption{Input}
    \end{subfigure}
    \begin{subfigure}[b]{0.24\textwidth}
        \includegraphics[width=\textwidth]{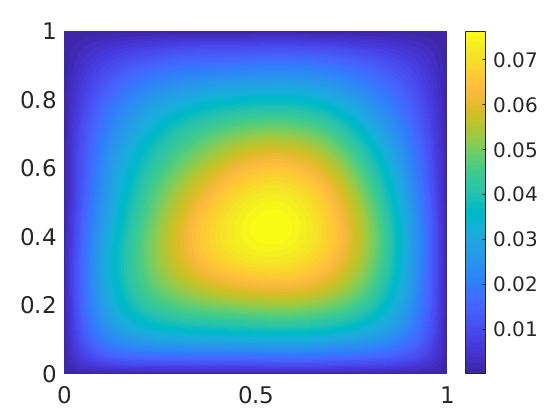}
        \caption{Ground Truth}
    \end{subfigure}
    \begin{subfigure}[b]{0.24\textwidth}
        \includegraphics[width=\textwidth]{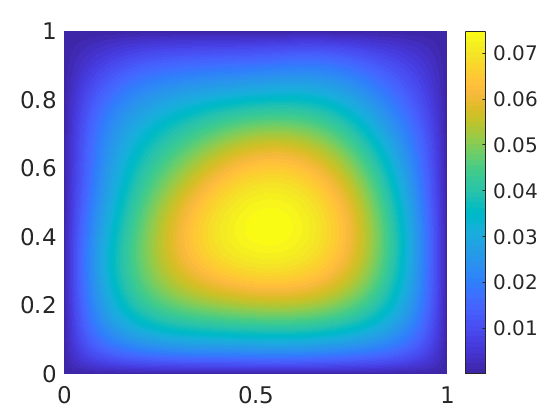}
        \caption{Approximation}
    \end{subfigure}
    \begin{subfigure}[b]{0.24\textwidth}
        \includegraphics[width=\textwidth]{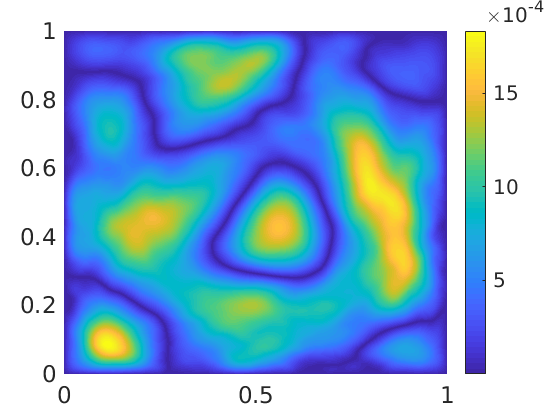}
        \caption{Error}
    \end{subfigure}

    \begin{subfigure}[b]{0.24\textwidth}
    \includegraphics[width=\textwidth]{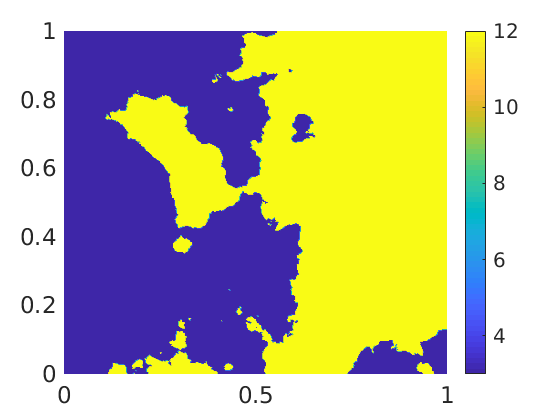}
    \caption{Input}
    \end{subfigure}
    \begin{subfigure}[b]{0.24\textwidth}
        \includegraphics[width=\textwidth]{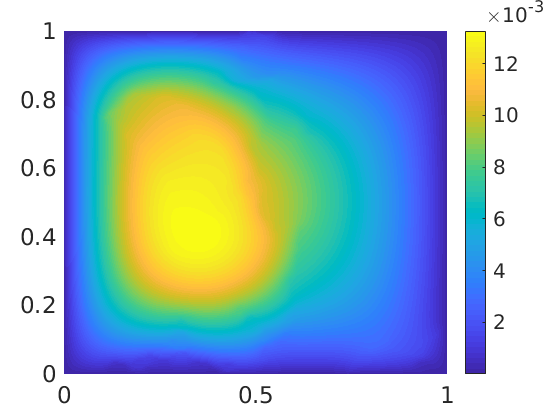}
        \caption{Ground Truth}
    \end{subfigure}
    \begin{subfigure}[b]{0.24\textwidth}
        \includegraphics[width=\textwidth]{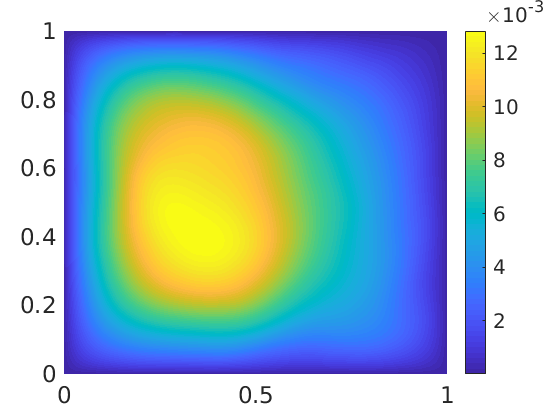}
        \caption{Approximation}
    \end{subfigure}
    \begin{subfigure}[b]{0.24\textwidth}
        \includegraphics[width=\textwidth]{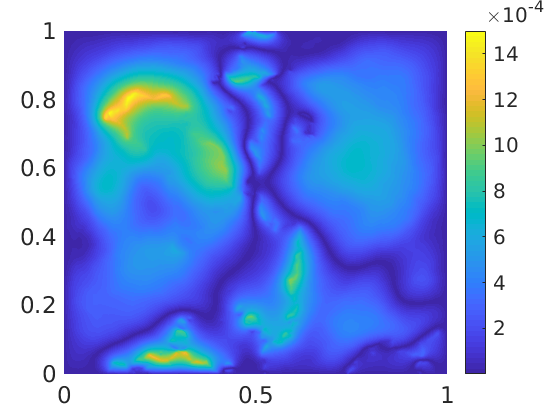}
        \caption{Error}
    \end{subfigure}

    \begin{subfigure}[b]{0.24\textwidth}
    \includegraphics[width=\textwidth]{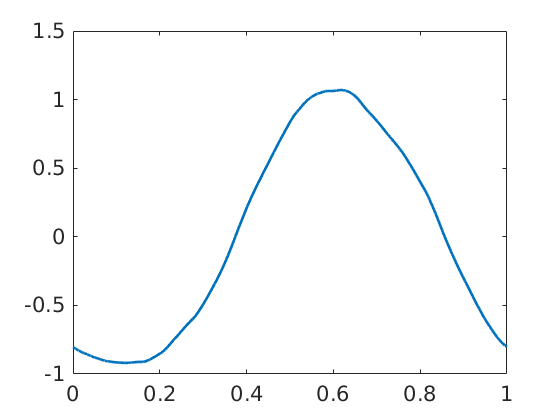}
    \caption{Input}
    \end{subfigure}
    \begin{subfigure}[b]{0.24\textwidth}
        \includegraphics[width=\textwidth]{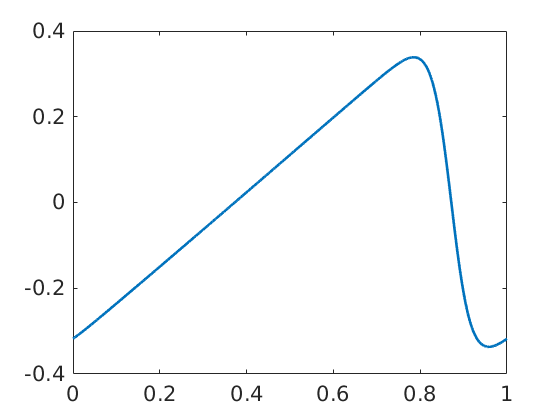}
        \caption{Ground Truth}
    \end{subfigure}
    \begin{subfigure}[b]{0.24\textwidth}
        \includegraphics[width=\textwidth]{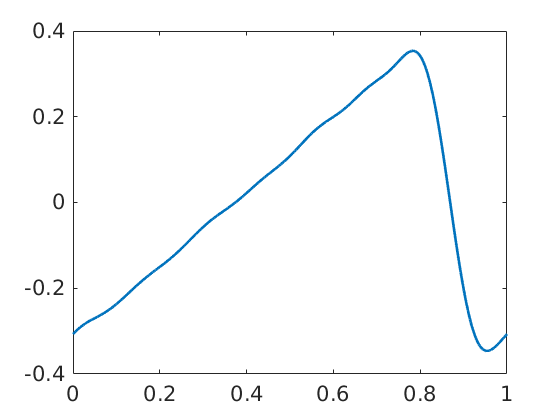}
        \caption{Approximation}
    \end{subfigure}
    \begin{subfigure}[b]{0.24\textwidth}
        \includegraphics[width=\textwidth]{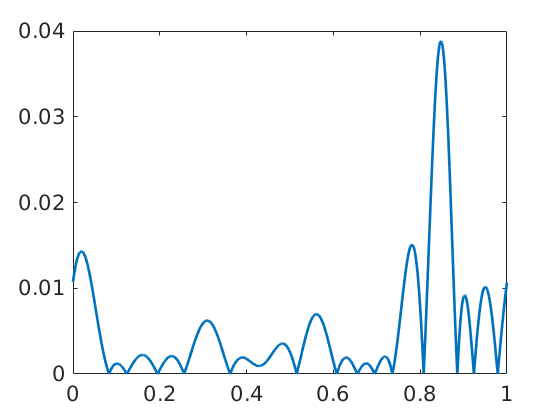}
        \caption{Error}
    \end{subfigure}

    \caption{Randomly chosen examples from the test set for each of the five considered problems.
    Each row is a different problem: linear elliptic, Poisson, Darcy flow with log-normal coefficients, 
    Darcy flow with piecewise constant coefficients, and Burgers' equation respectively from top to bottom. The approximations 
    are constructed with our best performing method (for \(N=1024\)): Linear \(d=150\), Linear \(d=150\), NN \(d=70\),
    NN \(d=70\), NN \(d=15\) respectively from top to bottom.}
    \label{fig:examples}
\end{figure}

We now present a series of numerical experiments that demonstrate
the effectiveness of our proposed methodology in the context
of the approximation of parametric PDEs. We work in settings
which both verify our theoretical results and show that the ideas
work outside the confines of the theory. 
The key idea underlying our work is to construct the neural network
architecture so that it is defined \emph{as a map between Hilbert spaces}
and only then to discretize and obtain a method that is implementable
in practice; prevailing methodologies first discretize and then
apply a standard neural network.
Our approach leads, when discretized, to methods that have properties which
are uniform with respect to the mesh size used.  We demonstrate this through
our numerical experiments.
In practice, we obtain an approximation $\Psinum$ to
$\Psin$, reflecting the numerical discretization used, and the fact that $\mu$
and its pushforward under $\Psi$ are only known to us through samples
and, in particular, samples of the pushforward of $\mu$ under the numerical 
approximation of the input-output map. However since, as we will show,
our method is robust to the discretization used, we will not explicitly
reflect the dependence of the numerical method in the notation that 
appears in the remainder of this section.

In Subsection \ref{ssec:PDES} we introduce a class of parametric 
elliptic PDEs arising from the Darcy model of flow in porous media, as well as the time-dependent, parabolic, Burgers' equation,
that define a variety of input-output maps for
our numerical experiments; we also introduce the probability measures
that we use on the input spaces.
Subsection \ref{sec:numlip} presents numerical results for a Lipschitz map. Subsections \ref{sec:numdarcy}, \ref{sec:burgers}
present numerical results for the Darcy flow problem and the flow map for the Burgers'
equation; this leads to
non-Lipschitz input-output maps, beyond our theoretical developments.
We emphasize that while our method is designed for approximating nonlinear 
operators \(\Psi\), we include some numerical examples where \(\Psi\) is linear. 
Doing so is helpful for confirming some of our theory and comparing against other 
methods in the literature. Note that
when \(\Psi\) is linear, each piece in the approximate decomposition \eqref{eq:apsipca} 
is also linear, in particular, \(\varphi\) is linear. Therefore it is sufficient to parameterize 
\(\varphi\) as a linear map (matrix of unknown coefficients) instead of a neural network. 
We include such experiments in Section \ref{sec:numlip} revealing that,
while a neural network approximating 
\(\varphi\) arbitrarily well exists, the optimization methods used for training the neural network
fail to find it. It may therefore be beneficial to directly build into the parametrization 
known properties of \(\varphi\), such as linearity, when they are known.
We emphasize that, for general nonlinear maps, linear methods
significantly underperform in comparison with our neural network
approximation and we will demonstrate this for the Darcy flow problem,
and for Burgers' equation.

We use standard implementations of PCA, with dimensions specified
for each computational example below.  All
computational examples use an identical neural network
architecture: a $5$-layer dense network with layer widths
$500, 1000, 2000, 1000, 500$, ordered from first to last layer, 
and the SELU nonlinearity \cite{selu}. 
We note that Theorem \ref{thm:approximation} 
requires greater depth for greater
accuracy but that we have found our $5$-layer network to
suffice for all of the examples described here. Thus we have not 
attempted to optimize the architecture of the neural network.
We use stochastic gradient descent with Nesterov momentum (\(0.99\)) to 
train the network parameters \cite{deeplearningbook}, 
each time picking the largest learning rate
that does not lead to blow-up in the error. While the network 
must be re-trained for each new choice of reduced dimensions \(d_\mathcal{X}, d_\mathcal{Y}\),
initializing the the hidden layers with a pre-trained network can 
help speed up convergence.

\subsection{PDE Setting}
\label{ssec:PDES}

We will consider a variety of solution maps defined
by second order elliptic PDEs of the form \eqref{eq:darcy}.
which are prototypical of many scientific applications. 
We take \(\OOmega = (0,1)^2\) to be the unit box, 
\(a \in L^\infty(\OOmega;\R_+), f \in L^2(\OOmega;\R)\), 
and let \(u \in H^1_0(\OOmega;\R)\) be the unique weak solution of \eqref{eq:darcy}. 
Note that, since $D$ is bounded, \(L^\infty(\OOmega;\R_+)\) is continuously embedded within the Hilbert space \(L^2(\OOmega;\R_+).\)
We will consider two variations of the input-output map generated
by the solution operator for \eqref{eq:darcy}; in one, it is Lipschitz and lends
itself to  the theory of Subsection \ref{sec:approxanalysis}  and, in the other,
it is not Lipschitz. We obtain numerical results which demonstrate our
theory as well as demonstrating
the effectiveness of our proposed methodology in the non-Lipschitz
setting.

Furthermore, we consider the one-dimensional viscous Burgers' equation 
on the torus given as
\begin{align}
\begin{split}
  \label{eq:burgers}
  \frac{\partial}{\partial t} u(\s,t) + \frac{1}{2}\frac{\partial}{\partial s} (u(\s,t))^2 &= \beta \frac{\partial^2}{\partial s^2} u(\s,t), \qquad (\s,t) \in \mathbb{T}^1 \times (0,\infty)   \\
  u(\s,0) &= u_0(\s), \qquad \qquad \:\:\:\: \s \in \mathbb{T}^1 
\end{split}
\end{align}
where \(\beta > 0\) is the viscosity coefficient and \(\mathbb{T}^1\) is the one dimensional unit torus obtained by equipping
the interval \([0,1]\) with periodic boundary conditions. We take \(u_0 \in L^2(\mathbb{T}^1;\R)\) and have that, for any \(t>0\), 
\(u(\cdot,t) \in H^r(\mathbb{T}^1;\R)\) for any \(r > 0\) is the unique weak 
solution to \eqref{eq:burgers} \cite{temam2012infinite}. In Subsection
\ref{sec:burgers}, we consider the input-output map generated by the flow map of \eqref{eq:burgers} evaluated at a fixed time
which is a locally Lipschitz operator.

We make use of four probability measures which we now describe.
The first, which will serve as a base measure in two dimensions, is the Gaussian 
\(\mug = \mathcal{N}(0, (-\Delta + 9I)^{-2})\)
with a zero Neumann boundary condition on the operator \(\Delta\).
Then we define \(\mul\) to be the log-normal measure defined as the push-forward 
of \(\mug\) under the exponential map i.e. $\mul = \exp_\sharp \mug$.
Furthermore, we define \(\mup = T_\sharp \mug \) to be the 
push-forward of \(\mug\) under the piecewise constant map
\begin{equation*}
  T(s) = \left\{
    \begin{aligned}
      &12 && s \ge 0,\\
      & 3 && s < 0.
    \end{aligned}
    \right.
  \end{equation*}
Lastly, we consider the Gaussian \(\mu_{\text{B}} = \mathcal{N}(0, 7^4(-\frac{d^2}{ds^2} + 7^2I)^{-2.5})\) defined on \(\mathbb{T}^1\). Figure \ref{fig:samples} shows an example draw from each of the above  measures.
We will use as $\mu$ one of these four measures in each experiment we conduct.
Such probability measures are commonly used in the stochastic modeling of 
physical phenomenon \cite{Lord}. For example, \(\mup\) may be thought of as modeling
the permeability of a porous medium containing two different constituent
parts \cite{darcyref}. Note that it is to be expected that a good
choice of architecture will depend on the probability measure used to
generate the inputs. Indeed good choices of the reduced dimensions 
$\dX$ and $\dY$ are determined
by the input measure and its pushforward under $\Psi$, respectively.

For each subsequently described problem we use, unless stated otherwise, \(N=1024\) training examples from $\mu$
and its pushforward under $\Psi$, from which
we construct $\Psin$, and then 
\(5000\) unseen testing examples from $\mu$ 
in order to obtain a Monte Carlo estimate of the relative test error:
\[\E_{x \sim \mu} \frac{\|(G_2 \circ \chi \circ F_1)(x) - \Psi(x)\|_\Y}{\|\Psi(x)\|_\Y}.\]
For problems arising from \eqref{eq:darcy}, all data is collected on a uniform \(421 \times 421\) mesh and the PDE is solved with a 
second order finite-difference scheme. For problems arising from \eqref{eq:burgers}, all data is collected on a uniform \(4096\) point mesh
and the PDE is solved using a pseudo-spectral method.  
Data for all other mesh sizes is sub-sampled from the original. 
We refer to the size of the discretization in one direction e.g. 421, as the \textit{resolution}.
We fix \(d_{\mathcal{X}} = d_{\mathcal{Y}}\) (the dimensions after PCA in the input and
output spaces) and refer to this as \textit{the reduced dimension}.
We experiment with using a linear map as well as a dense neural network 
for approximating \(\varphi\); in all figures we distinguish between these
by referring to Linear or NN approximations respectively. When parameterizing 
with a neural network, we use the aforementioned stochastic gradient based method for training,
while, when parameterizing with a linear map, we simply solve the linear least squares problem 
by the standard normal equations.

We also compare all of our results to the work of \cite{surrogatemodeling}
which utilizes a $19$-layer fully-connected convolutional neural network, 
referencing this approach as Zhu within the text. 
This is done to show that the image-to-image regression approach that many such works utilize 
yields approximations that are not consistent in the continuum, 
and hence across different discretizations; in contrast, our methodology
is designed as a mapping between Hilbert spaces and as a consequence
is robust across different discretizations.
For some problems in Subsection \ref{sec:numlip}, we compare to the method 
developed in \cite{cohenalgo}, which we refer to as Chkifa.
For the problems in Subsection \ref{sec:numdarcy},
we also compare to the reduced basis method 
\cite{DeVoreReducedBasis,quarteroni2015reduced} when instantiated with PCA.
We note that both Chkifa and the reduced basis method are intrusive, 
i.e., they need knowledge of the governing PDE. Furthermore the method of Chkifa needs
full knowledge of the generating process of the inputs. 
We re-emphasize that our proposed method is fully data-driven.

\subsection{Globally Lipschitz Solution Map}
\label{sec:numlip}

\begin{figure}[t]
    \centering
    \begin{subfigure}[b]{0.32\textwidth}
        \includegraphics[width=\textwidth]{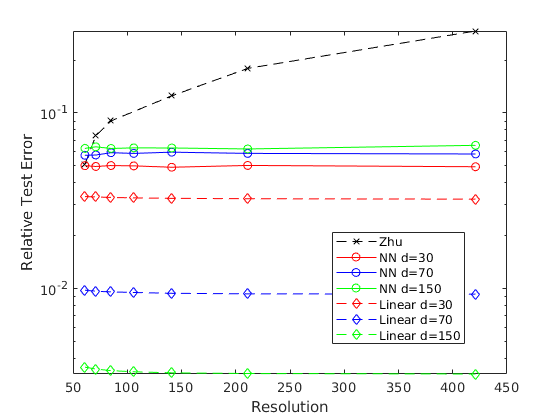}
        \caption{}
    \end{subfigure}
    \begin{subfigure}[b]{0.32\textwidth}
        \includegraphics[width=\textwidth]{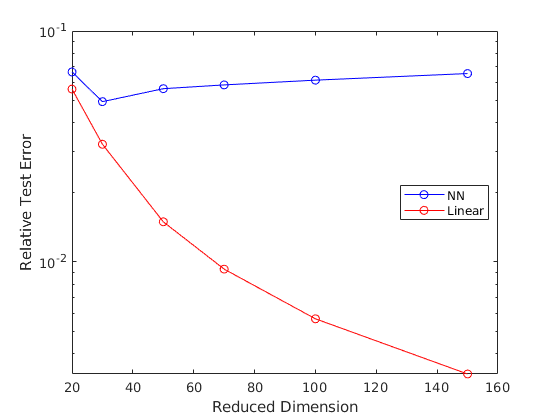}
        \caption{}
    \end{subfigure}
    \begin{subfigure}[b]{0.32\textwidth}
        \includegraphics[width=\textwidth]{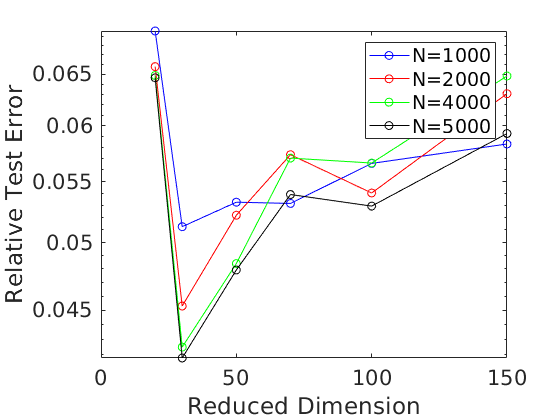}
        \caption{}
    \end{subfigure}

    \caption{Relative test errors on the linear elliptic problem. Using \(N=1024\) training examples, panel (a)
    shows the errors as a function of the resolution while panel (b) fixes a \(421 \times 421\) mesh and shows
    the error as a function of the reduced dimension. Panel (c) only shows results for our method using a neural network, 
    fixing a \(421 \times 421\) mesh and showing the error as a function of the reduced dimension for 
    different amounts of training data. }
    \label{fig:ellipticproblem}
\end{figure}

\begin{figure}[t]
    \centering
    \begin{subfigure}[b]{0.32\textwidth}
        \includegraphics[width=\textwidth]{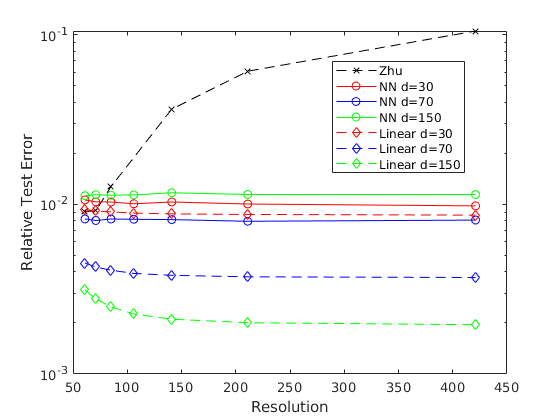}
        \caption{}
    \end{subfigure}
    \begin{subfigure}[b]{0.32\textwidth}
        \includegraphics[width=\textwidth]{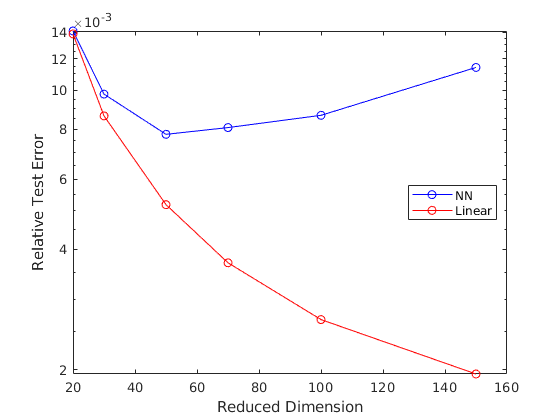}
        \caption{}
    \end{subfigure}
    \begin{subfigure}[b]{0.32\textwidth}
        \includegraphics[width=\textwidth]{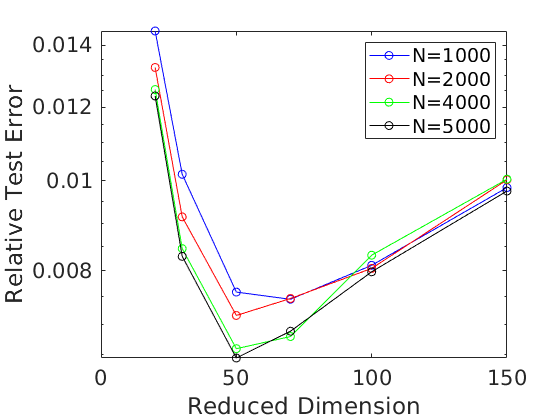}
        \caption{}
    \end{subfigure}

    \caption{Relative test errors on the Poisson problem. Using \(N=1024\) training examples, panel (a)
    shows the errors as a function of the resolution while panel (b) fixes a \(421 \times 421\) mesh and shows
    the error as a function of the reduced dimension. Panel (c) only shows results for our method using a neural network, 
    fixing a \(421 \times 421\) mesh and showing the error as a function of the reduced dimension for 
    different amounts of training data.}
    \label{fig:poissonproblem}
\end{figure}

\begin{figure}[t]
    \centering
    \begin{subfigure}[b]{0.32\textwidth}
        \includegraphics[width=\textwidth]{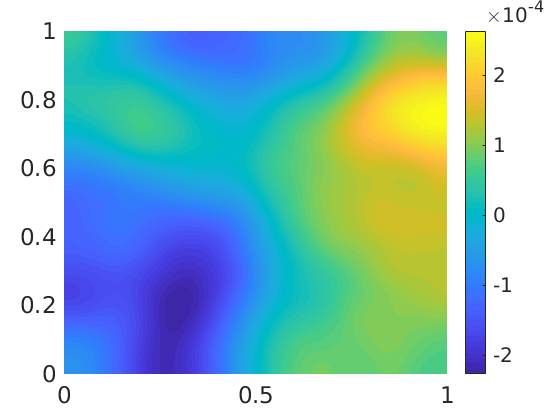}
        \caption{}
    \end{subfigure}
    \begin{subfigure}[b]{0.32\textwidth}
        \includegraphics[width=\textwidth]{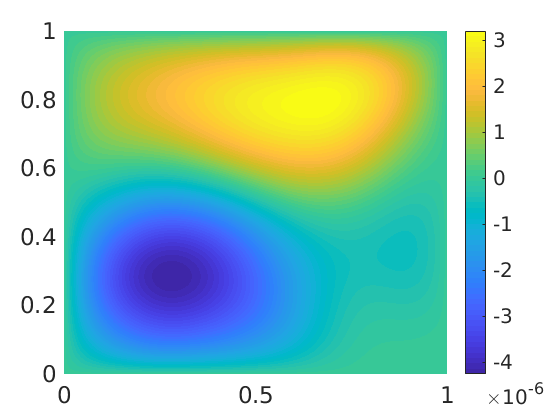}
        \caption{}
    \end{subfigure}
    \begin{subfigure}[b]{0.32\textwidth}
        \includegraphics[width=\textwidth]{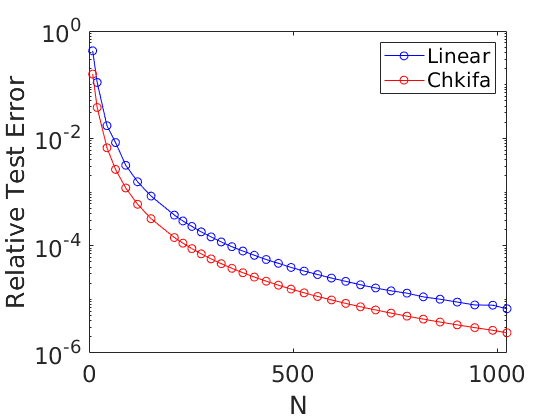}
        \caption{}
    \end{subfigure}

    \caption{Panel (a) shows a sample drawn from the model \eqref{eq:cohenmodel} while panel (b) shows the solution of the Poisson 
    equation with the sample from (a) as the r.h.s. Panel (c) shows the relative test error as a function of the amount of PDE solves/ training data 
    for the method of Chkifa and our method respectively. We use the reduced dimension \(d=N\).} 
    \label{fig:cohen}
\end{figure}

We consider the input-output map 
\(\Psi: L^2(\OOmega;\R) \to H^1_0(\OOmega;\R)\)
mapping \(f \mapsto u\) in \eqref{eq:darcy}
with the coefficient \(a\) fixed. Since \eqref{eq:darcy} is a linear PDE, \(\Psi\) is linear and therefore 
Lipschitz. We study two instantiations of this problem. In the first, we draw a single \(a \sim \mup\)
and fix it. We then solve \eqref{eq:darcy} with data \(f \sim \mug\).
We refer to this as 
the \textit{linear elliptic} problem. See row 1 of Figure \ref{fig:examples} 
for an example. In the second, we set \(a(w) = 1\) \(\forall w \in \OOmega\), in which case 
\eqref{eq:darcy} becomes the Poisson equation which we solve with data \(f \sim \mu = \mug\). We refer 
to this as the \textit{Poisson} problem. See row 2 of Figure \ref{fig:examples} for an example. 

Figure \ref{fig:ellipticproblem} (a) shows the relative test errors
as a function of the resolution on the linear elliptic problem, while Figure \ref{fig:poissonproblem} (a) shows them on the Poisson problem. The primary
observation to make about panel (a) in these two figures is that it shows that
the error in our proposed method does not change as the resolution changes. In contrast, it
also shows that the image-to-image regression approach of Zhu \cite{surrogatemodeling}, whilst accurate at low mesh resolution, fails to be invariant 
to the size of the discretization and errors increase in an uncontrolled fashion
as greater resolution is used. The fact that our dimension reduction approach 
achieves constant error as we refine the mesh, reflects its design as a 
method on Hilbert space which
may be approximated consistently on different meshes.
Since the operator \(\Psi\) here is linear, the true 
map of interest \(\varphi\) given by \eqref{eq:approxphi}
is linear since \(F_{\mathcal{Y}}\) and \(G_{\mathcal{X}}\) are, by the definition of PCA, linear.
It is therefore
unsurprising that the linear approximation consistently outperforms 
the neural network, a fact also demonstrated in panel (a) of the two 
figures.
While it is theoretically possible to find a neural network 
that can, at least, match the performance of the linear map,
in practice, the non-convexity of the associated optimization problem can cause non-optimal 
behavior. Panels (b) of Figures \ref{fig:ellipticproblem} and \ref{fig:poissonproblem} show the 
relative error as a function of the reduced dimension for a fixed mesh size. We see that while 
the linear maps consistently improve with the reduced dimension, the neural networks struggle 
as the complexity of the optimization problem is increased. This problem can usually be alleviated 
with the addition of more data as shown in panels (c), but there are still no guarantees that the optimal neural network is found. 
Since we use a highly-nonlinear 5-layer network to represent the linear \(\varphi\), this issue is exacerbated for
this problem and the addition of more data only slightly improves the accuracy as seen in panels (c).
In Appendix~\ref{app:error}, we show the relative test error during the training process and observe that 
some overfitting occurs, indicating that the optimization problem is stuck in a local minima away 
from the optimal linear solution.
This is an issue that is 
inherent to most deep neural network based methods. Our results suggest 
that building in \textit{a priori} information about the solution map, such as linearity, 
can be very beneficial for the approximation scheme as it can help reduce the 
complexity of the optimization.

To compare to the method of Chkifa \cite{cohenalgo}, we will assume the following model for the inputs,
\begin{equation}
\label{eq:cohenmodel}
f = \sum_{j=1}^\infty \xi_j \phi_j\
\end{equation}
where \(\xi_j \sim U(-1,1)\) is an i.i.d. sequence, and \(\phi_j = \sqrt{\lambda_j} \psi_j \) where \(\lambda_j, \psi_j\)
are the eigenvalues and eigenfunctions of the operator  \((-\Delta + 100I)^{-4.1}\) with a zero Neumann boundary.
This construction ensures that there exists \(p \in (0,1)\) such that \((\|\phi_j\|_{L^\infty})_{j \geq 1} \in \ell^p(\mathbb{N};\R)\)
which is required for the theory in \cite{cohenanalytic}.
We assume this model for \(f\), the r.h.s. of the Poisson equation, and consider the solution 
operator \(\Psi : \ell^\infty(\mathbb{N};\R) \to H_0^1(\OOmega; \R)\) mapping \((\xi_j)_{j \geq 1} \mapsto u\).
Figure \ref{fig:cohen} panels (a)-(b) show an example input from \eqref{eq:cohenmodel} and its 
corresponding solution \(u\).
Since this operator is linear, its Taylor series representation simply amounts to
\begin{equation}
\label{eq:taylorsum}
\Psi((\xi_j)_{j \geq 1}) = \sum_{j=1}^\infty \xi_j \eta_j
\end{equation}
where \(\eta_j \in H_0^1(\OOmega;\R)\) satisfy
\[-\Delta \eta_j = \phi_j.\]
This is easily seen by plugging in our model \eqref{eq:cohenmodel}
for \(f\) into the Poisson equation and formally inverting the Laplacian.
We further observe that the \(\ell^1(\mathbb{N};\R)\) summability of the sequence \((\|\eta_j\|_{H_0^1})_{j \geq 1}\) 
(inherited from \((\|\phi_j\|_{L^\infty})_{j \geq 1} \in \ell^p(\mathbb{N};\R)\)) implies that our power series \eqref{eq:taylorsum} is 
summable in \(H_0^1(\OOmega;\R)\). Combining the two observations yields analyticity of \(\Psi\) with the same rates as in \cite{cohenanalytic}
obtained via Stechkin's inequality. For a proof, see Theorem \ref{thm:poissonanalytic}.

We employ the method of Chkifa simply by truncation of \eqref{eq:taylorsum} to \(d\) elements, noting that in this simple linear 
setting there is no longer a need for greedy selection of the index set. We note that this truncation requires \(d\)
PDE solves of the Poisson equation hence we compare to our method when 
using \(N=d\) data points, since this also counts the number
of PDE solves. Since the problem is linear,
we use a linear map to interpolate the PCA latent spaces and furthermore set the reduced dimension of our PCA(s) to \(N\).
Panel (c) of Figure \ref{fig:cohen} shows the results. We see that the method 
of Chkifa outperforms our method for any fixed number of PDE solves, 
although the empirical rate of convergence appears very similar for both 
methods.  Furthermore we highlight that while our
  method appears to have a larger error constant than that of Chkifa, 
it has the advantage that it requires no knowledge of the model 
\ref{eq:cohenmodel} or of the Poisson equation; it is driven entirely 
by the training data.

\subsection{Darcy Flow}
\label{sec:numdarcy}

\begin{figure}[t]
    \centering
    \begin{subfigure}[b]{0.32\textwidth}
        \includegraphics[width=\textwidth]{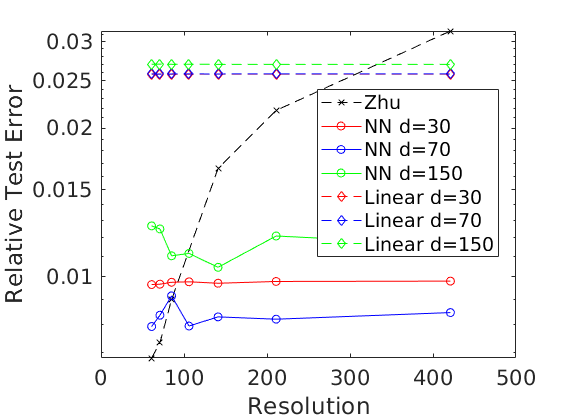}
        \caption{}
    \end{subfigure}
    \begin{subfigure}[b]{0.32\textwidth}
        \includegraphics[width=\textwidth]{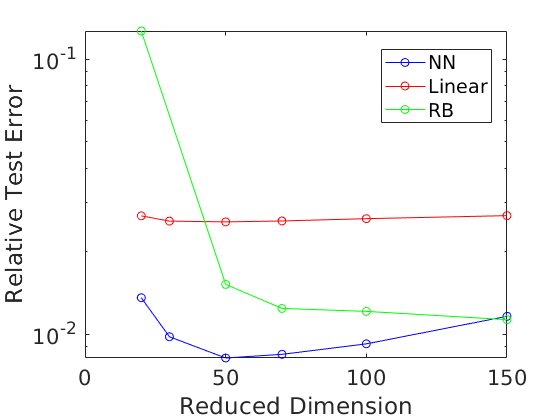}
        \caption{}
    \end{subfigure}
    \begin{subfigure}[b]{0.32\textwidth}
        \includegraphics[width=\textwidth]{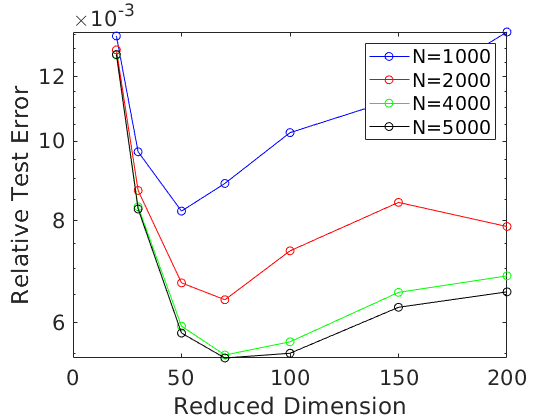}
        \caption{}
    \end{subfigure}

    \caption{Relative test errors on the Darcy flow problem with log-normal coefficients. Using \(N=1024\) training examples, panel (a)
    shows the errors as a function of the resolution while panel (b) fixes a \(421 \times 421\) mesh and shows
    the error as a function of the reduced dimension. Panel (c) only shows results for our method using a neural network, 
    fixing a \(421 \times 421\) mesh and showing the error as a function of the reduced dimension for 
    different amounts of training data.}
    \label{fig:lognormaldarcy}
\end{figure}

\begin{figure}[t]
    \centering
    \begin{subfigure}[b]{0.32\textwidth}
        \includegraphics[width=\textwidth]{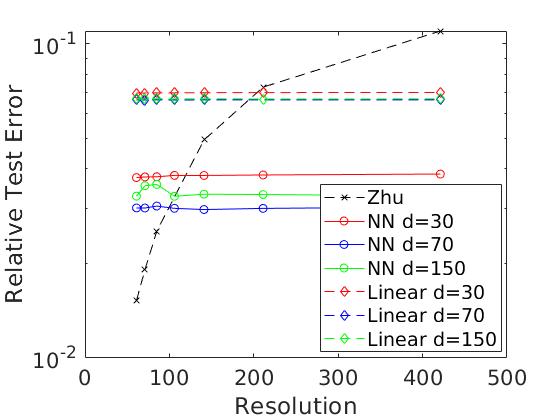}
        \caption{}
    \end{subfigure}
    \begin{subfigure}[b]{0.32\textwidth}
        \includegraphics[width=\textwidth]{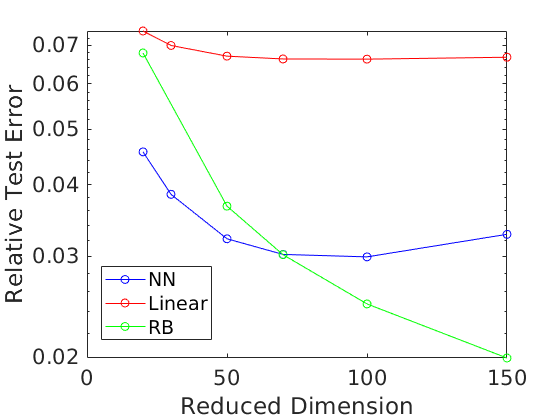}
        \caption{}
    \end{subfigure}
    \begin{subfigure}[b]{0.32\textwidth}
        \includegraphics[width=\textwidth]{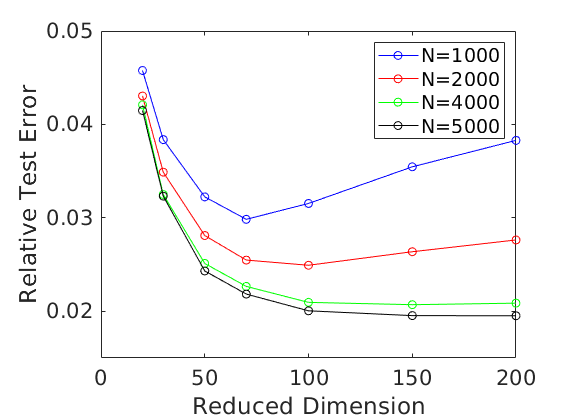}
        \caption{}
    \end{subfigure}

    \caption{Relative test errors on the Darcy flow problem with piecewise constant coefficients. Using \(N=1024\) training examples, panel (a)
    shows the errors as a function of the resolution while panel (b) fixes a \(421 \times 421\) mesh and shows
    the error as a function of the reduced dimension. Panel (c) only shows results for our method using a neural network, 
    fixing a \(421 \times 421\) mesh and showing the error as a function of the reduced dimension for 
    different amounts of training data.}
    \label{fig:piececonstdarcy} 
\end{figure}

\begin{figure}[t] 
    \centering
    \begin{subfigure}[b]{0.49\textwidth}
        \includegraphics[width=\textwidth]{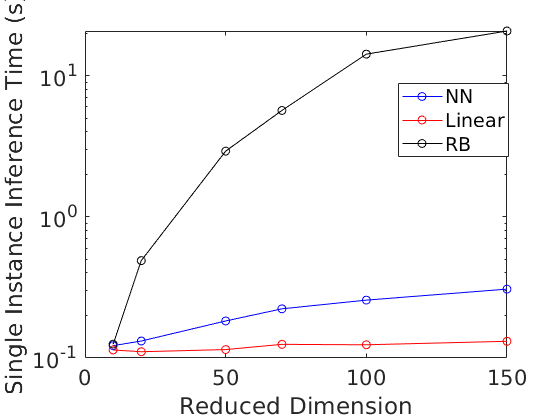}
        \caption{Online}
    \end{subfigure}
    \begin{subfigure}[b]{0.49\textwidth}
        \includegraphics[width=\textwidth]{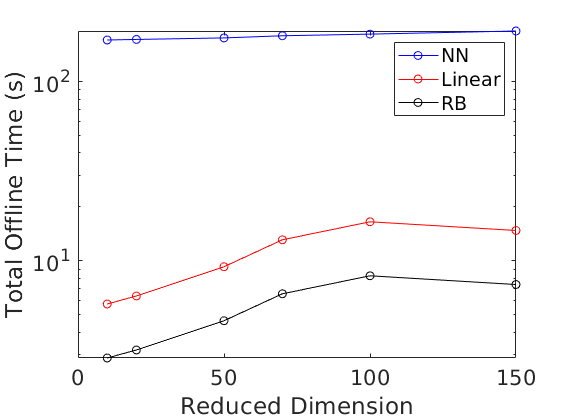}
        \caption{Offline}
    \end{subfigure}

    \caption{The online and offline computation times for the Darcy flow problem with piecewise constant coefficients.
    The number of training examples \(N=1024\) and grid resolution \(421 \times 421\) are fixed. The results are 
    reported in seconds and all computations are done on a single GTX 1080 Ti GPU. }
    \label{fig:timing}
\end{figure}

\begin{figure}[t]
    \centering
    \begin{subfigure}[b]{0.49\textwidth}
        \includegraphics[width=\textwidth]{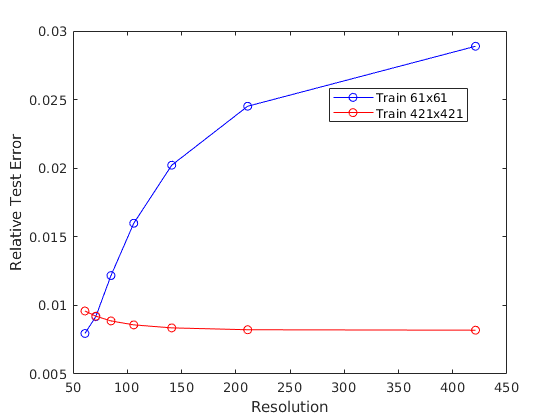}
        \caption{\(\mu_{\text{L}}\)}
    \end{subfigure}
    \begin{subfigure}[b]{0.49\textwidth}
        \includegraphics[width=\textwidth]{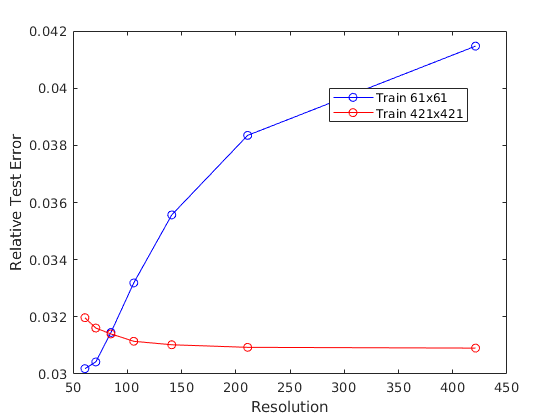}
        \caption{\(\mu_{\text{P}}\)}
    \end{subfigure}

    \caption{Relative test errors on both Darcy flow problems with reduced dimension \(d=70\), training on a single mesh and transferring the solution 
    to other meshes. When the training mesh is smaller than the desired output mesh, the PCA basis are interpolated using 
    cubic splines. When the training mesh is larger than the desired output mesh, the PCA basis are sub-sampled.}
    \label{fig:meshinterp}
\end{figure}

We now consider the input-output map
\(\Psi: L^\infty(\OOmega;\R_+) \to H^1_0(\OOmega;\R)\) mapping \(a \mapsto u\) in \eqref{eq:darcy}
with \(f(\s) = 1\) \(\forall \s \in \OOmega\) fixed. In this setting, 
the solution operator is nonlinear and is locally Lipschitz 
as a mapping from \(L^\infty(\OOmega;\R_+)\) 
to \(H^1_0(\OOmega;\R)\) \cite{DHS12}. 
However our results require a Hilbert space structure,
and we view the solution operator as a mapping from 
\(L^2(\OOmega;\R_+) \supset L^\infty(\OOmega;\R_+)\) into \(H^1_0(\OOmega;\R_+)\), noting that we will choose the probability measure $\mu$ on \(L^2(\OOmega;\R_+)\)
to satisfy $\mu(L^\infty(\OOmega;\R_+))=1.$
In this setting, $\Psi$ is not locally Lipschitz and 
hence Theorem \ref{thm:approximation} is
not directly applicable. Nevertheless, our methodology exhibits competitive numerical performance.
See rows 3 and 4 of Figure \ref{fig:examples} for an example.

Figure \ref{fig:lognormaldarcy} (a)
shows the relative test errors as a function of the resolution when \(a \sim \mu = \mul\) is log-normal while 
Figure \ref{fig:piececonstdarcy} (a) shows them when \(a \sim \mu = \mup\) is piecewise constant. In both settings,
we see that the error in our method is invariant to mesh-refinement. Since the problem is nonlinear,
the neural network outperforms the linear map. However we see the same issue 
as in Figure \ref{fig:ellipticproblem} where 
increasing the reduced dimension does not necessarily improve the error due to the increased complexity of the optimization problem. Panels (b) of Figures \ref{fig:lognormaldarcy} and  
\ref{fig:piececonstdarcy} confirm this observation. This issue can be alleviated with additional training data. Indeed,
 panels (c) of Figures \ref{fig:lognormaldarcy} and \ref{fig:piececonstdarcy} show that the error curve is flattened with 
 more data. We highlight that these results are consistent with our interpretation of
Theorem~\ref{thm:limit}: the reduced dimensions $\dX, \dY$ are 
determined first by the properties of the measure $\mu$ and its
pushforward,  and then the amount of data necessary is obtained to
ensure that the finite data approximation error is of the same
order of magnitude as the finite-dimensional approximation error.
In summary, the size of the training dataset $N$ should increase with the
number of reduced dimensions.

For this problem, we also compare to the reduced basis method (RB) when instantiated with 
PCA. We implement this by a standard Galerkin projection, expanding the solution in the PCA basis and 
using the weak form of \eqref{eq:darcy} to find the coefficients. We note that the errors of both methods are very close, but we find that the online runtime of our method is significantly better.
Letting \(K\) denote the mesh-size and \(d\) the reduced dimension, the reduced basis method has a runtime of \(\mathcal{O}(d^2K + d^3)\)
while our method has the runtime \(\mathcal{O}(dK)\) plus the runtime of the neural network which, in practice, we have found to be 
negligible. We show the online inference time as well as the offline training time of the methods in 
Figure \ref{fig:timing}. While the neural network has the highest offline cost, its small online cost 
makes it a more practical method. Indeed, without parallelization when \(d=150\), the total time (online and offline) 
to compute all 5000 test solutions is around 28 hours for the RBM. 
On the other hand, for the
neural network, it is 28 minutes. The difference is pronounced
when needing to compute many solutions in parallel. Since most modern architectures 
are able to internally parallelize matrix-matrix multiplication, the total time to train and compute the 5000 examples 
for the neural network is only 4 minutes. This issue can however be slightly alleviated 
for the reduced basis method with more stringent multi-core parallelization. We note that the linear map
has the lowest online cost and only a slightly worse offline cost 
than the RBM. This makes it the most 
suitable method for linear operators such as those presented in Section \ref{sec:numlip} or for applications 
where larger levels of approximation error can be tolerated.

We again note that the image-to-image regression approach of \cite{surrogatemodeling} does not scale 
with the mesh size. We do however acknowledge that for the small meshes for which the method was designed,
it does outperform all other approaches. This begs the question of whether one can design neural networks 
that match the performance of image-to-image regression but remain invariant with respect to the size 
of the mesh. The contemporaneous work \cite{neuralopour} takes a step in this direction.  

Lastly, we show that our method also has the ability to transfer a solution learned on one mesh to another.
This is done by interpolating or sub-sampling both of the input and output PCA basis from the 
training mesh to the desired mesh. Justifying this requires a smoothness assumption 
on the PCA basis; we are, however, not aware of any such results and believe this is an 
interesting future direction. The neural network is fixed and does not need to be re-trained on a new mesh.
We show this in Figure \ref{fig:meshinterp} for both Darcy flow problems. We note that when training on a 
small mesh, the error increases as we move to larger meshes, reflecting the interpolation error of the basis.
Nevertheless, this increase is rather small: as shown in Figure \ref{fig:meshinterp}, we obtain a \(3\%\) and 
a \(1\%\) relative error increasing when transferring solutions trained on a \(61 \times 61\) grid 
to a \(421 \times 421\) grid on each respective Darcy flow problem. On the other hand, when 
training on a large mesh, we see almost no error increase on the small meshes. This indicates that the 
neural network learns a property that is intrinsic to the solution operator and independent of the
discretization. 

\subsection{Burgers' Equation}
\label{sec:burgers}

\begin{figure}[t]
    \centering
    \begin{subfigure}[b]{0.32\textwidth}
        \includegraphics[width=\textwidth]{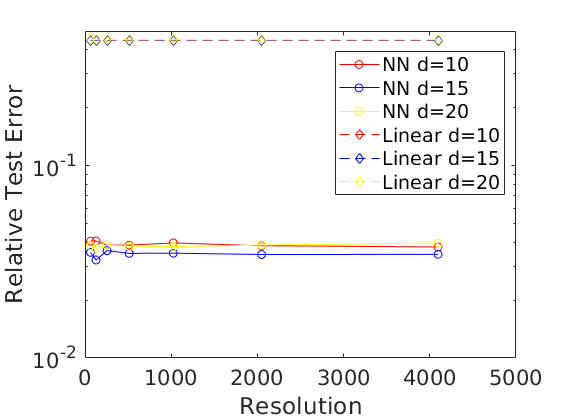}
        \caption{}
    \end{subfigure}
    \begin{subfigure}[b]{0.32\textwidth}
        \includegraphics[width=\textwidth]{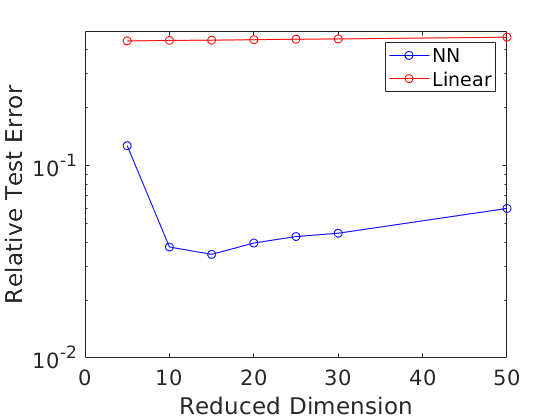}
        \caption{}
    \end{subfigure}
    \begin{subfigure}[b]{0.32\textwidth}
        \includegraphics[width=\textwidth]{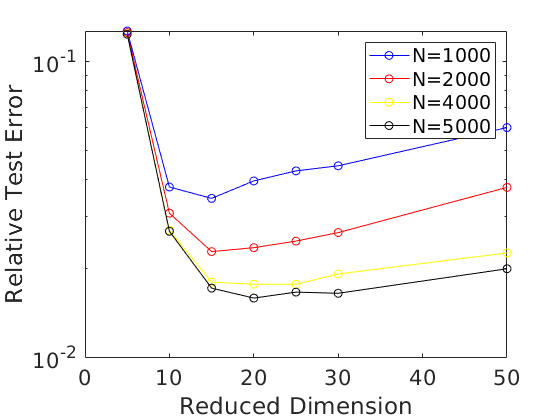}
        \caption{}
    \end{subfigure}

    \caption{Relative test errors on the Burgers' Equation problem. Using \(N=1024\) training examples, panel (a)
    shows the errors as a function of the resolution while panel (b) fixes a \(4096\) mesh and shows
    the error as a function of the reduced dimension. Panel (c) only shows results for our method using a neural network, 
    fixing a \(4096\) mesh and showing the error as a function of the reduced dimension for 
    different amounts of training data.}
    \label{fig:burgers}
\end{figure}

We now consider the input-output map \(\Psi :  L^2(\mathbb{T}^1;\R) \to H^r(\mathbb{T}^1;\R)\) 
mapping \(u_0 \mapsto u|_{t=1}\) in \eqref{eq:burgers} with \(\beta = 10^{-2}\) fixed. In this setting,
\(\Psi\) is nonlinear and locally Lipschitz but since we do not know the precise Lipschitz constant as defined in Appendix \ref{app:approxanalysis-local-Lipschitz},
we cannot verify that the assumptions of Theorem \ref{thm:approximation-local-lip} hold; nevertheless the numerical results
demonstrate the effectiveness of our methodology.  We take \(u_0 \sim \mu = \mu_{\text{B}}\); see rows 5 of Figure \ref{fig:examples} for an example.
Figure \ref{fig:burgers} (a) shows the relative test errors as a function of the resolution again demonstrating 
that our method is invariant to mesh-refinement. We note that, for this problem,
the linear map does significantly worse than the neural network in contrast to the Darcy flow
problem where the results were comparable. This is likely attributable
to the fact that the solution 
operator for Burgers' equation is more strongly nonlinear. As before, we observe from Figure \ref{fig:burgers} panel (b)
that increasing the reduced dimension does not necessarily improve the error due to the 
increased complexity of the optimization problem. This can again be mitigated by increasing the 
volume of training data, 
as indicated in Figure \ref{fig:burgers}(c); the curve of error
versus reduced dimension is flattened as \(N\) increases.

\section{Conclusion}
\label{sec:conclusion}
In this paper, we proposed a general data-driven methodology that can be used 
to learn mappings between separable Hilbert spaces. We proved 
consistency of the 
approach when instantiated with PCA in the setting of globally Lipschitz forward maps. 
We demonstrated the desired mesh-independent properties of our approach on 
parametric PDE problems, showing good numerical performance even on problems outside 
the scope of the theory. 

This work leaves many interesting directions open for future research. To understand 
the interplay between the reduced dimension and the amount of data needed requires a 
deeper understanding of neural networks and their interaction with the optimization algorithms
used to produce the approximation architecture. Even if the optimal neural
network is found by that optimization procedure, the question of the
number of parameters needed to achieve a given level of accuracy,
and how this interacts with the choice of reduced dimensions 
$\dX$ and $\dY$ (choice of which is determined by the input space
probability measure),
warrants analysis in order to reveal
the computational complexity of the proposed approach.  Furthermore, 
the use of PCA limits the scope of problems that can be addressed
to Hilbert, rather than general Banach spaces; even in Hilbert space,
PCA may not be the optimal choice of dimension reduction. The development of 
autoenconders on function space is a promising direction that has the potential
to address these issues; it also has many potential 
applications that are not limited to deployment within the methodology
proposed here. Finally we also
wish to study the use of our methodology in more challenging PDE problems, 
such as those arising in materials science, as well as for time-dependent 
problems such as multi-phase flow in porous media.
Broadly speaking  we view our contribution as a first step 
in the development of methods that generalize the ideas 
and applications of neural networks by operating on, and between,
spaces of functions.








\section*{Acknowledgments}
The authors are grateful to Anima  Anandkumar, Kamyar Azizzadenesheli, 
Zongyi Li and Nicholas H. Nelsen for helpful discussions in the general area 
of neural networks for PDE-defined maps between
Hilbert spaces. The authors thank Matthew M. Dunlop for sharing his code
for solving elliptic PDEs and generating Gaussian random fields. The work is supported by MEDE-ARL funding 
(W911NF-12-0022). AMS is also
partially supported by NSF (DMS 1818977) and AFOSR (FA9550-17-1-0185).
BH is partially supported by a 
Von K{\'a}rm{\'a}n instructorship at the California Institute of Technology.

\bibliographystyle{plain}
\bibliography{biblio}

\appendix

\section{Neural Networks And Approximation (Locally Lipschitz Case)}
\label{app:approxanalysis-local-Lipschitz}

This extends the approximation theory of Subsection~\ref{sec:approxanalysis} to the
case of solution maps 
\(\Psi: \X \to \Y\) that are \(\mu\)-measurable and locally Lipschitz
in the following sense 
\begin{equation}
\label{eq:Ldot}
\forall x,z \in \X\quad \|\Psi(x) - \Psi(z)\|_\Y \leq L(x,z) \|x-z\|_\X.
\end{equation}
where the function $L: \X \times \X \to \R_+$ is symmetric in its arguments, i.e., $L(x,z) = L(z,x)$,
and for any fixed $w \in \X$ the function $L(\cdot, w): \X \to \R_+$ is
 $\mu$-measurable and non-decreasing in the sense that $L(s, w) \le L(x, w)$ if $\| x\|_\X \ge \| w\|_\X$. 
Note that this implies that \(\Psi\) is locally bounded: for any \(x \in \X\)
\[\|\Psi(x)\|_\Y \le  \|\Psi(0)\|_\Y + \|\Psi(x) - \Psi(0)\|_\Y \leq \|\Psi(0)\|_\Y + L(x, 0) \|x\|_\X.\]
Hence we deduce that the pushforward \(\Psi_\sharp \mu\) has bounded fourth moments
provided that $\E_{x \sim \mu} (L(x, 0) \| x\|_\X)^4 < +\infty$:
\[
  \begin{aligned}
  \E_{y \sim \Psi_\sharp \mu} \|y\|^4_\Y &= \int_\X \|\Psi(x)\|_\Y^4 d\mu(x)
  \leq \int_\X (\|\Psi(0)\|_\Y + L(x, 0) \|x\|_\X)^4 d\mu(x) \\
  & \le 2^3 \left( \| \Psi(0)\|_\Y + \int_\X L(x, 0)^4 \| x \|^4_\X d\mu(x) \right)   < \infty,
\end{aligned}
\]
where we used a generalized triangle inequality proven in \cite[Cor. 3.1]{takahasi2010refined}.

\begin{theorem}
  \label{thm:approximation-local-lip}
Let \(\X\), \(\Y\) be real, separable Hilbert spaces, $\Psi$ a mapping
from $\X$ into $\Y$ and let \(\mu\) be a probability measure
supported on \(\X\) such that 
\begin{equation*}
  \E_{x \sim \mu}  L(x,x)^2 < +\infty, \quad \E_{x \sim \mu} L( x, 0)^2 \|x \|_\X^2 < \infty,
\end{equation*}
where $L(\cdot,\cdot)$ is given in \eqref{eq:Ldot}.
Fix $\dX, \dY$, $N \ge \max\{\dX, \dY\},$ $\delta \in (0,1)$, $\tau>0$
and define \(M = \sqrt{\E_{x \sim \mu} \|x\|_\X^2 / \delta}\). Then 
there exists a constant $c(\dX, \dY) \ge 0$ and  neural network
\(\chi \in \mathcal{M}(\dX, \dY; t, r,M)\)
with $t \le c ( \log (\sqrt{\dY}/ \tau) +1 )$ layers and
$r \le c (\epsilon^{-\dX} \log ( \sqrt{\dY}/\tau) + 1)$ 
active weights and biases in each component, so that 
\begin{equation}
\label{eq:approximationbound-loc-lip}
\begin{aligned}
\E_{\{x_j\} \sim \mu} \E_{x \sim \mu}\bigl(\en(x)\bigr) &\le C \Bigg( \tau + \sqrt{\delta} \\
& +  \left( \sqrt{\frac{\dX}{N}} + R^{\mu}(V_{\dX}^{\X})  \right)^{1/2}
+  \left( \sqrt{\frac{\dY}{N}} + R^{\Psi_\sharp \mu}(V_{\dY}^{\Y}) \right)^{1/2} \Bigg),
\end{aligned}
\end{equation}
where \(\en(x) := \|\Psin(x) - \Psi(x)\|_\Y\) and
$C > 0$ is independent of $\dX, \dY, N, \delta$ and $\epsilon$.
\end{theorem}

\begin{proof}
  Our method of proof is similar to the proof of Theorem~\ref{thm:approximation} and
  for this reason we shorten some of the arguments. 
  Recall the constant $Q$ from Theorem \ref{thm:pca_generalization_bound}. In
what follows we take $Q$ to be the maximum of the two such constants
when arising from application of the theorem on the two different probability
spaces $(\X,\mu)$ and $(\Y, \Psi_\sharp \mu)$. As usual we employ the shorthand notation $\E$
to denote $\E_{\{x_j \} \sim \mu}$ the expectation with respect to the dataset $\{ x_j\}_{j=1}^N$.

We begin by approximating the error incurred by using \(\Psip\) given by
\eqref{eq:apsipca}:
\begin{align}
\label{eq:truephiapprox-local-lip}
\begin{split}
  \E & \E_{x \sim \mu} \|\Psip(x) - \Psi(x)\|_\Y
  = \E \E_{x \sim \mu} \|(\GY \circ \FY \circ \Psi \circ \GX \circ \FX)(x) - \Psi(x)\|_\Y \\
&= \E \E_{x \sim \mu} \left\|\Pi_{V_{\dY,N}^\Y} \Psi (\Pi_{V_{\dX,N}^\X} x) - \Psi(x) \right\|_\Y \\
&\leq \E \E_{x \sim \mu} \left\|\Pi_{V_{\dY,N}^\Y} \Psi(\Pi_{V_{\dX,N}^\X} x) - \Pi_{V_{\dY,N}^\Y} \Psi(x) \right\|_\Y + \E \E_{x \sim \mu}\left \|\Pi_{V_{\dY,N}^\Y} \Psi(x) - \Psi(x) \right\|_\Y \\
& \leq \E \E_{x \sim \mu} L(\Pi_{V_{\dX,N}^\X} x, x)  \left\|\Pi_{V_{\dX,N}^\X} x - x \right\|_\X + \E \E_{y \sim \Psi_\sharp \mu} \left\|\Pi_{V_{\dY,N}^\Y} y - y \right\|_\Y,
\end{split}
\end{align}
noting that the operator norm of an orthogonal projection is \(1\).
Now since $L(\cdot, x)$ is non-decreasing we infer that $L(\Pi_{V_{\dX,N}^\X} x, x) \le L(x,x)$
and then using Cauchy-Schwarz we obtain
\begin{align}
  \begin{split}
    \E \E_{x \sim \mu} \|\Psip(x) - \Psi(x)\|_\Y
    & \leq \left(\E_{x \sim \mu} |L(x,x)|^2 \right)^{1/2}
    \left( \E \E_{x \sim \mu} \left\|\Pi_{V_{\dX,N}^\X} x - x \right\|_\X^2 \right)^{1/2} \\
    &\:\:\:\:+ \E \E_{y \sim \Psi_\sharp \mu} \left\|\Pi_{V_{\dY,N}^\Y} y - y \right\|_\Y,\\
    &= L' \left( \E R^{\mu}(V_{\dX,N}^\X) \right)^{1/2} 
 + \left( \E [R^{\Psi_\sharp \mu}(V_{\dY,N}^\Y)] \right)^{1/2}
\end{split}
\end{align}
where we used H{\"o}lder's inequality in the last line and defined the new constant 
$
L' := \left(\E_{x \sim \mu} |L(x,x)|^2 \right)^{1/2}.
$
Theorem \ref{thm:pca_generalization_bound}
allows us to control this error, and leads to 
\begin{align}
\label{eq:truephiapprox2-local-lip}
\begin{split}
\E \E_{x \sim \mu} \en & \le \E \E_{x \sim \mu}\|\Psin(x) - \Psip(x)\|_\Y+ \E \E_{x \sim \mu} \|\Psip(x) - \Psi(x)\|_\Y\\
&\le  \E_{x \sim \mu}\|\Psin(x) - \Psip(x)\|_\Y\\
&\quad+ L' \left( \sqrt{\frac{Q \dX}{N}} + R^{\mu}(V_{\dX}^\X) \right)^{1/2}+
\left(\sqrt{\frac{Q \dY}{N}} + R^{\Psi_\sharp \mu}(V_{\dY}^\Y)\right)^{1/2}.
\end{split}
\end{align}

We now approximate \(\varphi\) by a neural network \(\chi\) as before. 
To that end we first note from Lemma \ref{l:add} that 
\(\varphi\) is locally Lipschitz, and hence continuous, as a mapping from $\R^{\dX}$
into $\R^{\dY}.$ 
Identify the components \(\varphi(s) = (\varphi^{(1)}(s), \dots, \varphi^{(\dY)}(s))\)
where each function \(\varphi^{(j)} \in C(\R^{\dX};\R)\).
We consider the restriction of each component function to the set  \([-M,M]^{\dX}\).

By \cite[Thm.~1]{yarotsky} and using the same arguments as in the proof of
Theorem~\ref{thm:approximation} there exist neural networks
\(\chi^{(1)}, \dots, \chi^{(\dY)} : \R^{\dX} \to \R\),
$\chi^{(j)} \in \mathcal{M}(\dX; t^{(j)}, r^{(j)})$, with layer and active weight parameters $t^{(j)}$ and $r^{(j)}$ satisfying 
$ t^{(j)} \le c^{(j)} [ \log( M\sqrt{\dY} / \tau) + 1 ],$ and 
 $ r^{(j)} \le c^{(j)}( \tau/2M)^{- \dX} [ \log( M\sqrt{\dY} /\tau ) +1  ]$
with constants $c^{(j)}(\dX) >0$, so that
\begin{equation*}
  \big| \big( \chi^{(1)}(s), \dots, \chi^{(\dY)}(s) \big)  - \varphi(s) \big|_2 < \tau
  \quad \forall s \in [-M,M]^{\dX}.
\end{equation*}

We  now simply 
define \(\chi: \R^{\dX} \to \R^{\dY}\) as the  stacked network
\((\chi^{(1)},\dots,\chi^{\dY})\) extended by zero outside of $[-M, M]^{\dX}$
to  immediately obtain
\begin{equation}
\label{eq:nnepsilonclose-local-lip}
\sup_{ s \in [-M,M]^{\dX}}
\big| \chi(s)  - \varphi(s) \big|_2 < \tau. 
\end{equation}
Thus, by construction 
$\chi \in \mathcal{M}(\dX, \dY, t, r, M)$ with at most
$t \le \max_j t^{(j)}$ many layers and
$r \le r^{(j)}$ 
many active weights and biases in each of its components.

Define the set 
$A = \{x \in \X : \FX(x) \in [-M,M]^{\dX}\}.$
By Lemma \ref{lemma:bounded_projection} below, \(\mu(A) \geq 1 - \delta\) and
\(\mu(A^c) \leq \delta\).  Define the approximation error 
$\ep(x) := \|\Psin(x) - \Psip(x)\|_\Y$
and decompose its expectation as 
\[\E_{x \sim \mu}[\ep(x)] = \int_A \ep(x) d\mu(x) + \int_{A^c} \ep(x) d\mu(x)=: I_A + I_{A^c}.\]
For the first term,
\begin{align}
\label{eq:erroronA-local-lip}
\begin{split}
  I_A &\leq  \int_A  \|(\GY \circ \chi \circ \FX)(x)
  - (\GY \circ \varphi \circ \FX)(x)\|_\Y d\mu(x) \leq  \tau, 
\end{split}
\end{align}
by using the fact, established in Lemma \ref{l:add}, that
\(\GY\) is Lipschitz with Lipschitz constant $1$, the
\(\tau\)-closeness of \(\chi\) 
to \(\varphi\) from \eqref{eq:nnepsilonclose-local-lip}, and \(\mu(A) \leq 1\). 
For the second term we have, using that
$\GY$ has Lipschitz constant $1$ and that $\chi$ takes value zero on $A^c$, 
\begin{align}
\label{eq:erroroutsideA-local-lip}
\begin{split}
I_{A^c} &\leq  \int_{A^c}  \|(\GY \circ \chi \circ \FX)(x) - (\GY \circ \varphi \circ \FX)(x)\|_\Y d\mu(x) \\ 
&\leq  \int_{A^c} | \chi(\FX(x)) - \varphi(\FX(x))|_2 d\mu(x)= \int_{A^c} | \varphi(\FX(x))|_2 d\mu(x).
\end{split}
\end{align}
Once more from Lemma \ref{l:add}  we have that
\[|\FX(x)|_2 \leq \|x\|_{\X}; \quad |\varphi(v)|_2 \leq |\varphi(0)|_2 + L(\GX(v), 0) |v|_2\] 
so that, using the hypothesis on $L$ and global Lipschitz property of $\FX$ and $\GX$ we can write
\begin{align}
\label{eq:erroroutsideA2-local-lip}
\begin{split}
  I_{A^c} & \leq \mu(A^c)|\varphi(0)|_2
  +   \int_{A^c}  L( x, 0) | \FX(x)|_2 d\mu(x) \\
  & \leq \mu(A^c) |\varphi(0)|_2
  + \int_{A^c}  L(x, 0) \| x \|_\X d\mu(x)\\
  & \le  \mu(A^c) |\varphi(0)|_2
  +  \mu(A^c)^{\frac12}
  \left( \E_{x \sim \mu}  L( x, 0)^2 \| x \|^2_\X\right)^{1/2} \\
 &\le  \delta |\varphi(0)|_2
 + \delta^{\frac12} L'' \\
 &\le \sqrt{\delta}(|\varphi(0)|_2 + L'')
\end{split}
\end{align}
where $L'' := \left( \E_{x \sim \mu}  L( x, 0)^2 \| x \|^2_\X\right)^{1/2}$. 
Combining \eqref{eq:truephiapprox2}, \eqref{eq:erroronA} 
and \eqref{eq:erroroutsideA2} we obtain the desired result. 
{}
\end{proof}




\section{Supporting Lemmas}
\label{app:supportlemmas}

In this Subsection we present and prove auxiliary lemmas that are used
throughout the proofs in the article.
The proof of Theorem \ref{thm:pca_generalization_bound}
made use of the following proposition, known as Fan's Theorem, 
proved originally in \cite{Fan}.  We state and prove it here in the 
infinite-dimensional setting as this generalization may be  of independent interest.
Our proof follows the steps
of Fan's original proof in the finite-dimensional setting. The work \cite{Fannew}, 
through which we first became aware of Fan's result, gives an elegant 
generalization; however it is unclear whether that approach is easily applicable
in  infinite dimensions due to issues of compactness.

\begin{lemma}[Fan \cite{Fan}]
\label{thm:fan}
Let \((\mathcal{H}, \langle \cdot, \cdot \rangle, \|\cdot\|)\) be a separable Hilbert space and
\(C : \mathcal{H} \to \mathcal{H}\) a non-negative, self-adjoint, compact operator. Denote by
\(\lambda_1 \geq \lambda_2 \geq \dots\) the eigenvalues of \(C\) and, for any \(d \in \mathbb{N} \setminus \{0\}\), let 
\(S_d\) denote the set of collections of \(d\) orthonormal elements of \(\mathcal{H}\). Then
\[\sum_{j=1}^d \lambda_j = \max_{\{u_1,\dots,u_d\} \in S_d} \sum_{j=1}^d \langle Cu_j, u_j \rangle.\]
\end{lemma}

\begin{proof}
Let \(\phi_1,\phi_2,\dots\) denote the orthonormal eigenfunctions of \(C\) corresponding to the 
eigenvalues \(\lambda_1,\lambda_2,\dots\) respectively. Note that for \(\{\phi_1,\dots,\phi_d\} \in S_d\), we have
\[\sum_{j=1}^d \langle C \phi_j, \phi_j \rangle = \sum_{j=1}^d \lambda_j \|\phi_j\|^2 = \sum_{j=1}^d \lambda_j.\]
Now let \(\{u_1,\dots,u_d\} \in S_d\) be arbitrary. Then for any \(j \in \{1,\dots,d\}\), we have
$u_j = \sum_{k=1}^\infty \langle u_j, \phi_k \rangle \phi_k$
and thus
\begin{align*}
\langle Cu_j, u_j \rangle &= \sum_{k=1}^\infty \lambda_k |\langle u_j, \phi_k \rangle|^2 \\
&= \lambda_d \sum_{k=1}^\infty |\langle u_j, \phi_k \rangle|^2 + \sum_{k=d+1}^\infty (\lambda_k - \lambda_d)|\langle u_j, \phi_k \rangle|^2 + \sum_{k=1}^d (\lambda_k - \lambda_d)|\langle u_j, \phi_k \rangle|^2.
\end{align*}
Since \(\|u_j\|^2 = 1\), we have 
$\|u_j\|^2 = \sum_{k=1}^\infty |\langle u_j, \phi_k \rangle|^2 = 1$
therefore 
\begin{align*}
\lambda_d \sum_{k=1}^\infty |\langle u_j, \phi_k \rangle|^2 + \sum_{k=d+1}^\infty (\lambda_k - \lambda_d)|\langle u_j, \phi_k \rangle|^2 &= \lambda_d \sum_{k=1}^d |\langle u_j, \phi_k \rangle|^2 + \sum_{k=d+1}^\infty \lambda_k |\langle u_j, \phi_k \rangle|^2 \\
&\leq \lambda_d \sum_{k=1}^\infty |\langle u_j, \phi_k \rangle|^2 = \lambda_d
\end{align*}
using the fact that \(\lambda_k \leq \lambda_d\), \(\forall k > d\). We have shown
$\langle Cu_j, u_j \rangle \leq \lambda_d + \sum_{k=1}^d (\lambda_k - \lambda_d) |\langle u_j, \phi_k \rangle|^2.$
Thus
\begin{align*}
\sum_{j=1}^d (\lambda_j - \langle Cu_j, u_j \rangle) &\geq \sum_{j=1}^d \left( \lambda_j - \lambda_d - \sum_{k=1}^d (\lambda_k - \lambda_d) |\langle u_j, \phi_k \rangle|^2 \right) \\
&= \sum_{j=1}^d (\lambda_j - \lambda_d) \left ( 1 - \sum_{k=1}^d |\langle u_k, \phi_j \rangle|^2 \right ).
\end{align*}
We now extend the finite set of $\{u_k\}_{k=1}^d$ from a $d-$dimensional
orthonormal set to an orthonormal basis $\{u_k\}_{k=1}^\infty$ for $\mathcal{H}$.
Note that \(\lambda_j \geq \lambda_d\),  \(\forall j \leq d\) and that 
\[\sum_{k=1}^d |\langle u_k, \phi_j \rangle|^2 \leq \sum_{k=1}^\infty |\langle u_k, \phi_j \rangle|^2 = \|\phi_j\|^2 = 1\]
therefore
$\sum_{j=1}^d (\lambda_j - \langle Cu_j, u_j \rangle) \geq 0$
concluding the proof.

\end{proof}

Theorem \ref{thm:pca_generalization_bound} relies on a Monte Carlo estimate of the Hilbert-Schmidt distance
between \(C\) and \(C_N\) that we state and prove  below.

\begin{lemma}
\label{lemma:convariancemontecarlo}
Let \(C\) be given by \eqref{eq:covaraince} and \(C_N\) by \eqref{eq:empiricalcovariance} then there exists a constant \(Q \geq 0\), depending only on \(\nu\), such that
\begin{equation*}
\E_{\{u_j\} \sim \nu} \|C_N - C\|_{HS}^2 = \frac{Q}{N}.
\end{equation*}
\end{lemma}

\begin{proof}
Define \(C^{(u_j)} := u_j \otimes u_j\) for any \(j \in \{1,\dots,N\}\) and \(C^{(u)} \coloneqq u \otimes u\) for any \(u \in \mathcal{H}\), noting that
\(\E_{u \sim \nu} [C^{(u)}] = C = \E_{u_j \sim \nu} [C^{(u_j)}]\).
Further we note that
\begin{equation*}
\E_{u \sim \nu} \|C^{(u)}\|^2_{HS} = \E_{u \sim \nu} \|u\|^4 < \infty
\end{equation*}
and, by Jensen's inequality,
$
\|C\|^2_{HS} \leq \E_{u \sim \nu} \|C^{(u)}\|^2_{HS} < \infty.
$
Once again using the shorthand notation $\E$ in place of  $\E_{\{u_j\} \sim \nu}$ 
we compute,
\begin{align*}
  \E \|C_N - C\|_{HS}^2
  &=  \E \| \frac{1}{N} \sum_{j=1}^N C^{(u_j)} - C \|^2_{HS}
  = \E \| \frac{1}{N} \sum_{j=1}^N C^{(u_j)} \|^2_{HS} - \|C\|^2_{HS} \\
  &= \frac{1}{N} \E_{u \sim \nu} \|C^{(u)}\|_{HS}^2 
    + \frac{1}{N^2} \sum_{j=1}^N \sum_{k \neq j}^N \langle \E[C^{(u_j)}], \E[C^{(u_k)}] \rangle_{HS} - \|C\|^2_{HS} \\
&= \frac{1}{N} \E_{u \sim \nu} \|C^{(u)}\|_{HS}^2 + \frac{N^2 - N}{N^2} \|C\|^2_{HS} - \|C\|^2_{HS} \\
  &= \frac{1}{N} \big ( \E_{u \sim \nu} \|C^{(u)}\|_{HS}^2 -  \|C\|^2_{HS} \big )
    = \frac{1}{N} \E_{u \sim \nu} \| C^{(u)} - C \|^2_{HS}.
\end{align*}
Setting \(Q = \E_{u \sim \nu} \| C^{(u)} - C \|^2_{HS}\) completes the proof.
\end{proof}

The following lemma, used in the proof of Theorem \ref{thm:approximation}, estimates Lipschitz constants
of various maps required in the proof.

\begin{lemma}
  \label{l:add} 
The maps \(\FX\), \(\FY\), \(\GX\) and \(\GY\) are globally Lipschitz:
\begin{align*}
 |\FX(v) - \FX(z)|_2 &\le \|v - z\|_\X, \quad \forall v,z \in \X\\
 |\FY(v) - \FY(v)|_2 &\le \|v - z\|_\Y, \quad \forall v,z \in \Y\\  
 \|\GX(v) - \GX(z)\|_\X &\le |v- z|_2, \quad \forall v,z \in \R^{\dX}\\  
 \|\GY(v) - \GY(z)\|_\X &\le  |v-z|_2, \quad \forall v,z \in \R^{\dY}.
\end{align*}
Furthermore, if $\Psi$ is locally Lipschitz and satisfies
\[\forall x,w \in \X\quad \|\Psi(x) - \Psi(w)\|_\Y \leq L(x,w) \|x-w\|_\X,\]
with  $L: \X \times \X \to \R_+$ that is symmetric with respect to its arguments, and increasing in the sense that
$L(s, w) \le L(x, w),$ if $ \| x \|_\X \ge \| s\|_\X$. 
Then
$\varphi$ is also locally Lipschitz and
\[|\varphi(v) - \varphi(z)|_2 \leq L(\GX(v), \GX(z)) |v-z|_2, \quad \forall v, z \in \R^{\dX}.\]
\end{lemma}

\begin{proof}
We establish that \(\FY\) and \(\GX\) are Lipschitz and
estimate the Lipschitz constants; the proofs for \(\FX\) and \(\GY\) are
similar. Let \(\phi_{1,N}^{\Y},\dots,\phi_{\dY,N}^{\Y}\) denote the eigenvectors of the empirical covariance with 
respect to the data \(\{y_j\}_{j=1}^N \) which span \(V_{\dY,N}^\Y\) and let \(\phi_{\dY+1,N}^{\Y}, \phi_{\dY+2,N}^{\Y}, \dots\)
be an orthonormal extension to \(\Y\). Then, by Parseval's identity,
\begin{align*}
|\FY(v) - \FY(z)|_2^2 &= \sum_{j=1}^{\dY} \langle v-z, \phi_{j,N}^{\Y} \rangle_{\Y} ^2 \leq  \sum_{j=1}^{\infty} \langle v-z, \phi_{j,N}^{\Y} \rangle_{\Y} ^2 = \|v-z\|_\Y^2.
\end{align*}
A similar calculation for $\GX$, using $\phi_{1,N}^{\X},\dots,\phi_{\dX,N}^{\X}$ the eigenvectors of the
  empirical covariance of the data $\{x_j\}_{j=1}^N$ yields
\begin{align*}
\|\GX(v) - \GX(z)\|_\X^2 =  \| \sum_{j=1}^{\dX} (v_j - z_j)\phi_{j,N}^{\X} \|_\X^2 
= \sum_{j=1}^{\dX} |v_j - z_j|^2
= |v-z|_2^2
\end{align*}
for any \(v,z \in \R^{\dX}\), using the fact that the empirical eigenvectors can be extended to an orthonormal basis for \(\X\). 
Recalling that \(\varphi = \FY \circ \Psi \circ \GX\), the above estimates immediately yield
\begin{equation*}
  \begin{aligned}
    |\varphi(v) - \varphi(z)|_2
    & \leq L(\GX(v), \GX(z)) |v-z|_2,\quad \forall v,z \in \R^{\dX}.
\end{aligned}
\end{equation*}
\end{proof}

The following lemma establishes a bound on the size of the set \(A\) 
that was defined in the proof of Theorems~\ref{thm:approximation} and \ref{thm:approximation-local-lip}.

\begin{lemma}
  \label{lemma:bounded_projection}
    Fix \(0 < \delta < 1\),  let \(x \sim \mu\) be a random variable and
    set \(M = \sqrt{\E_{x \sim \mu} \|x\|^2_\X / \delta}\). 
Define $\FX$ using the random dataset $\{ x_j \}_{j=1}^N \sim \mu$
then,
\begin{equation*}
  \mathbb{P} \left( \FX(x) \not\in [-M, M]^{\dX} \right) \le \delta,
\end{equation*}
where the probability is computed with  respect to both $x$  and the
 $ x_j$'s. 
\end{lemma}
\begin{proof}
  Denote by \(\phi_{1,N}^{\X},\dots,\phi_{\dX,N}^{\X}\) the orthonormal set used to define \(V^{\X}_{\dX,N}\)
  \eqref{eq:pcasubspace} and let \(\phi_{\dX + 1,N}^{\X},\phi_{\dX + 2,N}^{\X},\dots\)
be an orthonormal extension of this basis to \(\X\). For any \(j \in \{1,\dots, \dX\}\), by Chebyshev's inequality, we have
\[\Prob_{x \sim \mu}(|\langle x, \phi_{j,N}^{\X} \rangle_\X | \geq M) \leq \frac{\E_{x \sim \mu}|\langle x, \phi_{j,N}^{\X} \rangle_\X |^2}{M^2}.\]
Note that the expectation is taken only with respect to the randomness in \(x\)
and not \(\phi_{1,N}^{\X},\dots,\phi_{\dX,N}^{\X}\), so the right hand side 
is itself a random variable. We further compute 
\begin{align*}
  &\Prob_{x \sim \mu}(|\langle x, \phi_{1,N}^{\X} \rangle_\X | \geq M, \dots, |\langle x, \phi_{\dX,N}^{\X} \rangle_\X | \geq M) \\
  &\qquad \leq \frac{1}{M^2} \E_{x \sim \mu} \sum_{j=1}^{\dX} |\langle x, \phi_{j,N}^{\X} \rangle_\X|^2
    \leq \frac{1}{M^2} \E_{x \sim \mu} \sum_{j=1}^{\infty} |\langle x, \phi_{j,N}^{\X} \rangle_\X |^2
    = \frac{1}{M^2} \E_{x \sim \mu} \|x\|^2_\X
\end{align*}
noting that \(\|x\|_\X^2 = \sum_{j=1}^\infty |\langle x, \xi_j \rangle_\X|^2\) for any orthonormal basis \(\{\xi_j\}_{j=1}^\infty\) of \(\X\)
hence the randomness in \(\phi_{1,N}^{\X},\phi_{2,N}^{\X},\dots\) is inconsequential. 
Thus we find that,  with $\Prob$ denoting probability with
respect to both $x \sim \mu$ and the random
data used to define $\FX$,
\begin{align*}
\Prob(|\langle x, \phi_{1,N}^{\X} \rangle_\X | \leq M, \dots, |\langle x, \phi_{\dX,N}^{\X} \rangle_\X | \leq M) &\geq 1- \frac{1}{M^2} \E_{x \sim \mu} \|x\|^2_\X, = 1 - \delta 
\end{align*}
the desired result.
\end{proof}

\section{Analyticity of the Poisson Solution Operator}
\label{app:b}

Define \(\X = \{\xi \in \ell^\infty(\mathbb{N};\R) : \|\xi\|_{\ell^\infty} \leq 1\}\) and let \(\{\phi_j\}_{j=1}^\infty\)
be some sequence of functions with the property that \((\|\phi_j\|_{L^\infty})_{j \geq 1} \in \ell^p(\mathbb{N};\R)\)
for some \(p \in (0,1)\).
Define \(\Psi : \X \to H_0^1(\OOmega;\R)\) as mapping a set of coefficients \(\xi = (\xi_1,\xi_2,\dots) \in \X\) to \(u \in H_0^1(\OOmega;\R)\) the unique weak solution of 
\[-\Delta u = \sum_{j=1}^\infty \xi_j \phi_j \quad \text{in} \quad \OOmega, \qquad u|_{\partial \OOmega} = 0.\]
Note that since \(\OOmega\) is a bounded domain and \(\xi \in \X\), we have that 
\(\sum_{j=1}^\infty \xi_j \phi_j \in L^2(\Omega;\R)\) since our assumption implies \((\|\phi_j\|_{L^\infty})_{j \geq 1} \in \ell^1(\mathbb{N};\R)\). Therefore \(u\)
is indeed the unique weak-solution of the Poisson equation \cite[Chap. 6]{evans10}
and \(\Psi\) is well-defined.

\begin{theorem}
\label{thm:poissonanalytic}
Suppose that there exists \(p \in (0,1)\) such that \((\|\phi_j\|_{L^\infty})_{j \geq 1} \in \ell^p(\mathbb{N};\R)\). 
Then 
\[\lim_{K \to \infty} \sup_{\xi \in \X} \|\Psi(\xi) - \sum_{j=1}^K \xi_j \eta_j \|_{H_0^1} = 0\]

where, for each  \(j \in \mathbb{N}\), \(\eta_j \in H_0^1(\OOmega;\R)\) satisfies
\[-\Delta \eta_j =\phi_j \quad \text{in} \quad \OOmega, \qquad u|_{\partial \OOmega} = 0.\]

\end{theorem}

\begin{proof}
By linearity and Poincar{\'e} inequality, we obtain the Lipschitz estimate
\[\|\Psi(\xi^{(1)}) - \Psi(\xi^{(2)})\|_{H_0^1}
  \leq C \| \sum_{j=1}^\infty (\xi^{(1)}_j - \xi^{(2)}_j) \phi_j \|_{L^2}, \quad \forall \xi^{(1)}, \xi^{(2)} \in \X\] 
for some \(C > 0\).  Now let \(\xi = (\xi_1,\xi_2,\dots) \in \X\) be arbitrary 
and define the sequence \(\xi^{(1)} = (\xi_1,0,0,\dots), \xi^{(2)} = (\xi_1,\xi_2,0,\dots), \dots\). Note 
that, by linearity, for any \(K \in \mathbb{N}\), 
$\Psi(\xi^{(K)}) = \sum_{j=1}^K \xi_j \eta_j.$
Then, using our Lipschitz estimate, 
\begin{align*}
\|\Psi(\xi) - \Psi(\xi^{K})\|_{H_0^1} &\lesssim \|\sum_{j={K+1}}^\infty \xi_j \phi_j \|_{L^2} 
\lesssim \sum_{j=K+1}^\infty \|\phi_j\|_{L^\infty} 
\leq K^{1- \frac{1}{p}} \|(\|\phi_j\|_{L^\infty})_{j \geq 1}\|_{\ell^p}
\end{align*}
where the last line follows by Stechkin’s inequality \cite[Sec. 3.3]{cohenanalytic}. Taking the supremum over \(\xi \in \X\)
and the limit \(K \to \infty\) completes the proof.
\end{proof}

\section{Error During Training}
\label{app:error}

Figures \ref{fig:gaussianelliptic_training} and \ref{fig:coeffpoisson_training} show the relative test error computed during the training process for the problems 
presented in Section~\ref{sec:numlip}. For both problems, we observe a slight amount of overfitting{}
when more training samples are used and the reduced dimension is sufficiently large. This is because the true map of 
interest \(\varphi\) is linear while the neural network parameterization is highly non-linear hence 
more prone to overfitting larger amounts of data. While this suggests that simpler neural networks might perform better 
on this problem, we do not carry out such experiments as our goal is simply to show that building in \textit{a priori}
information about the problem (here linearity) can be beneficial as show in Figures \ref{fig:ellipticproblem} and \ref{fig:poissonproblem}.

\begin{figure}[h]
    \centering
    \begin{subfigure}[b]{0.24\textwidth}
        \includegraphics[width=\textwidth]{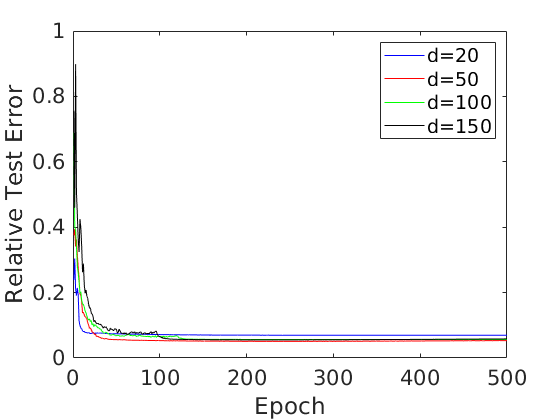}
        \caption{\(N=1000\)}
    \end{subfigure}
    \begin{subfigure}[b]{0.24\textwidth}
        \includegraphics[width=\textwidth]{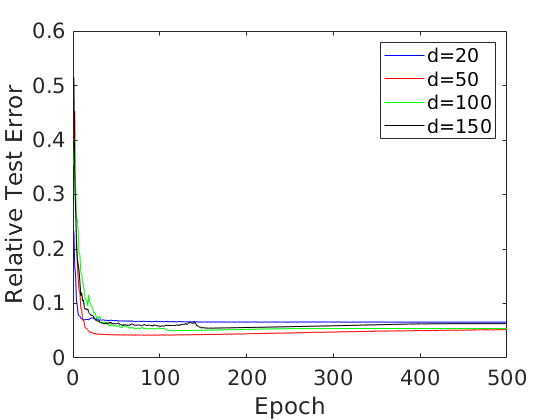}
        \caption{\(N=2000\)}
    \end{subfigure}
    \begin{subfigure}[b]{0.24\textwidth}
        \includegraphics[width=\textwidth]{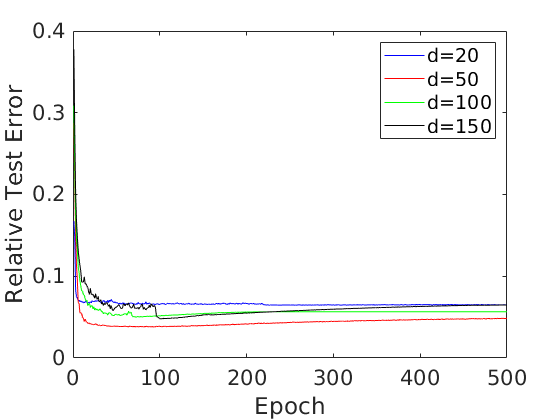}
        \caption{\(N=4000\)}
    \end{subfigure}
    \begin{subfigure}[b]{0.24\textwidth}
        \includegraphics[width=\textwidth]{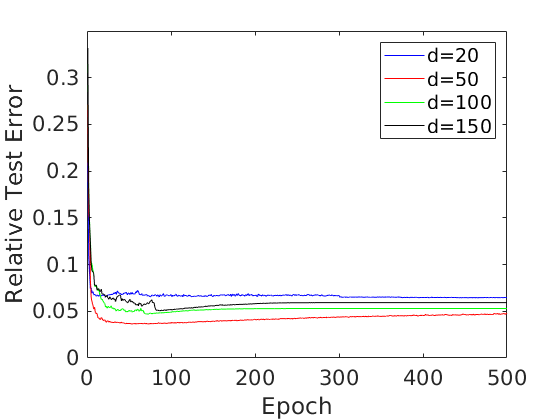}
        \caption{\(N=5000\)}
    \end{subfigure}

    \caption{Relative test errors as a function of the training epoch on the linear elliptic problem.
    The amount of training examples used is varied in each panel.}
    \label{fig:gaussianelliptic_training}
\end{figure}

\begin{figure}[h]
    \centering
    \begin{subfigure}[b]{0.24\textwidth}
        \includegraphics[width=\textwidth]{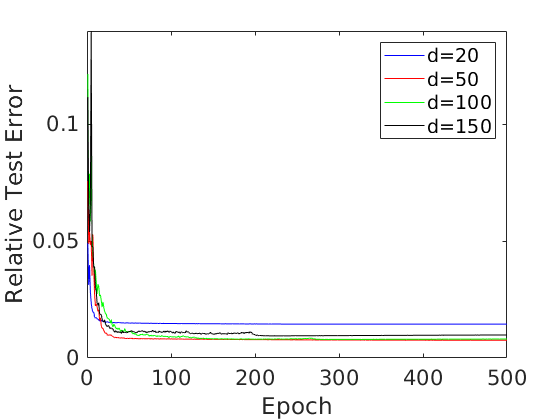}
        \caption{\(N=1000\)}
    \end{subfigure}
    \begin{subfigure}[b]{0.24\textwidth}
        \includegraphics[width=\textwidth]{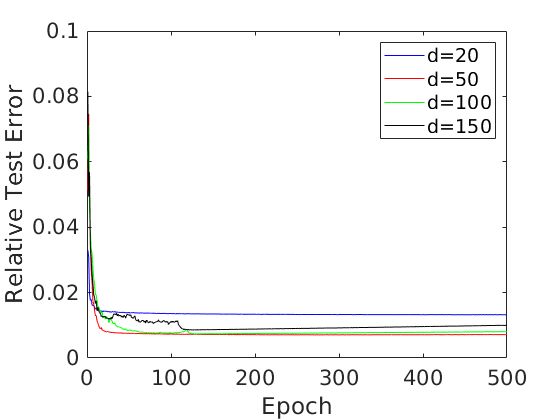}
        \caption{\(N=2000\)}
    \end{subfigure}
    \begin{subfigure}[b]{0.24\textwidth}
        \includegraphics[width=\textwidth]{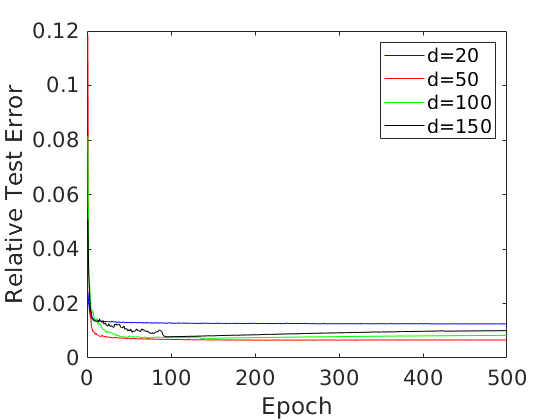}
        \caption{\(N=4000\)}
    \end{subfigure}
    \begin{subfigure}[b]{0.24\textwidth}
        \includegraphics[width=\textwidth]{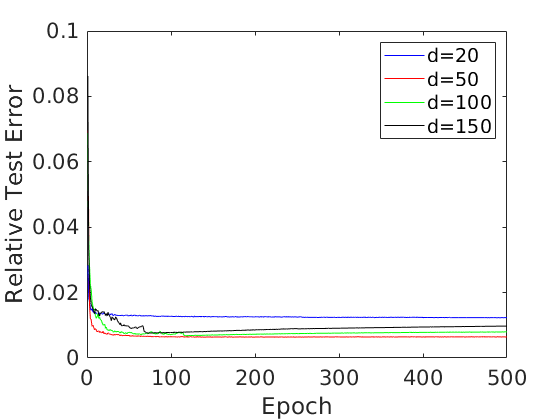}
        \caption{\(N=5000\)}
    \end{subfigure}

    \caption{Relative test errors as a function of the training epoch on the Poisson problem.
    The amount of training examples used is varied in each panel.}
    \label{fig:coeffpoisson_training}
\end{figure}

\end{document}